\documentclass[11pt]{aims}
\usepackage{amsmath, amssymb, mathrsfs}
  \usepackage{paralist}
  \usepackage{graphics} 
  \usepackage{epsfig} 
\usepackage{graphicx} 
 \usepackage{epstopdf}
 \usepackage[colorlinks=true]{hyperref}
 \hypersetup{urlcolor=blue, citecolor=blue}
 \usepackage[toc,page]{appendix}
\usepackage{chngcntr}
\usepackage{hyperref}

\usepackage{titlesec}
\titleformat{\section}{\Large\bfseries}{\thesection.}{4pt}{}
\titleformat{\subsection}{\large\bfseries}{\thesection.\arabic{subsection}.}{4pt}{}
\titleformat{\subsubsection}{\bfseries}{\thesection.\arabic{subsection}.\arabic{subsubsection}.}{4pt}{}
\titleformat*{\paragraph}{\bfseries}
\titleformat*{\subparagraph}{\bfseries}
\usepackage[margin=1in]{geometry}

  \def\currentyear{}
   \def\currentmonth{}
    \def\ppages{}
     \def\DOI{}

\newtheorem{theorem}{Theorem}[section]
\newtheorem{corollary}[theorem]{Corollary}

\newtheorem{lemma}[theorem]{Lemma}
\newtheorem{proposition}[theorem]{Proposition}
\theoremstyle{definition}
\newtheorem{definition}[theorem]{Definition}
\newtheorem{remark}[theorem]{Remark}

\numberwithin{equation}{section}

\title[A non-scaling invariant semilinear heat equation]{Construction of a stable blowup solution with a prescribed behavior for a non-scaling invariant semilinear heat equation}

\author[G. K. Duong, V. T. Nguyen, H. Zaag ]{}

\email[G. K. Duong]{duong@univ-paris13.fr}
\email[V. T. Nguyen]{Tien.Nguyen@nyu.edu}
\email[H. Zaag] {Hatem.Zaag@univ-paris13.fr}

\thanks{-----------------\\ \today}

\begin{document}
\maketitle

\centerline{\scshape Giao Ky Duong $^{\dagger , }$ \footnote{ G. K. Duong is fully funded by the European Union's Horizon 2020 research and innovation programme under the Marie Sk\l odowska-Curie grant agreement No 665850.}, Van Tien Nguyen$^\ast$ and Hatem Zaag$^{\dagger , }$ \footnote{ H. Zaag is supported by the  ANR  projet ANA\'E  ref. ANR -13-BS01-0010-03}}
\medskip
{\footnotesize
  \centerline{$^\dagger$ Universit\'e Paris 13, Sorbonne Paris Cit\'e, LAGA,  CNRS (UMR 7539), F-93430, Villetaneuse, France.}
\centerline{$^\ast$ New York University in Abu Dhabi, P.O. Box 129188, Abu Dhabi, United Arab Emirates.}
 
}

\begin{abstract} 
We consider  the semilinear heat equation
\begin{eqnarray*}
\partial_t u =  \Delta u + |u|^{p-1} u \ln ^{\alpha}( u^2  +2),
\end{eqnarray*}
in the whole space $\mathbb{R}^n$, where $p > 1$ and $ \alpha \in \mathbb{R}$. Unlike the standard case $\alpha = 0$, this equation is  not scaling invariant. We construct for this equation a solution which blows up in finite  time $T$ only at one blowup point $a$, according to the following  asymptotic dynamics:
\begin{eqnarray*}
u(x,t) \sim \psi(t) \left(1 + \frac{(p-1)|x-a|^2}{4p(T -t)|\ln(T -t)|} \right)^{-\frac{1}{p-1}} \text{ as } t \to T,
\end{eqnarray*} 
where $\psi(t)$ is the unique  positive  solution  of the ODE
\begin{eqnarray*}
\psi' = \psi^p \ln^{\alpha}(\psi^2  +2), \quad  \lim_{t\to T}\psi(t) = + \infty.
\end{eqnarray*}
The construction  relies on the reduction of the  problem  to a  finite  dimensional one  and a topological argument based  on the index theory to get the conclusion. By the interpretation of the parameters  of the finite dimensional  problem  in terms 	of the blowup time and the blowup point, we show  the stability  of the constructed  solution  with respect  to perturbations in  initial data. To our knowledge, this is the first successful construction for a genuinely  non-scale invariant  PDE of a stable  blowup solution with the derivation of the blowup profile. From this point of view, we consider our result as a breakthrough.
\end{abstract}

\maketitle
\section{Introduction.}

We are interested in the semilinear heat equation
\begin{equation}\label{equ:problem}
\left\{
\begin{array}{rcl}
\partial_t u &=& \Delta u + F(u), \\
u(0) &=& u_0 \in L^\infty(\mathbb{R}^n),
\end{array}
\right.
\end{equation}
where $u(t): \mathbb{R}^n \to \mathbb{R},$  $\Delta $ stands  for  the Laplacian  in $\mathbb{R}^n$ and
\begin{equation}\label{F-u}
F(u)  = |u|^{p-1} u \ln^{\alpha}(u^2  +2), \quad p > 1, \quad \alpha \in\mathbb{R}.
\end{equation}
By standard  results  the model  \eqref{equ:problem} is wellposed  in $L^{\infty}(\mathbb{R}^n)$ thanks to a fixed-point argument. More precisely, there is a  unique maximal  solution on $[0, T),$ with $T \leq +\infty$. If $T < +\infty$,  then   the solution of \eqref{equ:problem}  may develop  singularities  in  finite time $T$, in  the sense that
$$\| u(t)\|_{L^{\infty}} \to  + \infty \text{ as } t \to T.$$ 
In this case, $T $ is called the blowup time of $u$. Given $a \in \mathbb{R}^n$, we say that $a$ is a blowup point of $u$ if and only if  there exists $(a_j,t_j ) \to (a,T)$ as $j \to +\infty$ such that $|u(a_j,t_j)| \to +\infty$ as $j \to +\infty$.\\

In the special case  $\alpha = 0,$ the equation \eqref{equ:problem} becomes the standard semilinear heat equation
 \begin{equation}\label{euqa-problem-specical-alpha=0}
 \partial_t u = \Delta u + |u|^{p-1}u. 
 \end{equation}
This equation is invariant under the following scaling transformation
\begin{equation}\label{eq:inv}
u \mapsto u_\lambda(x,t):= \lambda^\frac{2}{p-1}u(\lambda x, \lambda^2 t).
\end{equation}
An extensive literature is devoted to equation \eqref{euqa-problem-specical-alpha=0}  and no rewiew can be exhaustive. Given our interest in the construction  question with a  prescribed blowup behavior, we  only mention previous work in this direction. 
 
In \cite{BKnon94}, Bricmont and Kupiainen   showed the existence of a solution of \eqref{euqa-problem-specical-alpha=0}  such that
 \begin{equation}\label{eq:1}
 \| (T - t)^{\frac{1}{p-1}} u(a + z \sqrt{(T -t) |\ln(T -t)|}, t) - \varphi_0(z)\|_{L^{\infty}(\mathbb{R}^n)} \to 0, \text{ as } t \to T,
 \end{equation}
where
$$\varphi_0 (z) = \left(  p-1 + \frac{(p-1)^2 z^2}{4p }\right)^{-\frac{1}{p-1}},$$
(note that Herrero and Vel\'azquez \cite{HVcpde92} proved the same result with a different method; note also that Bressan \cite{Brejde92} made a similar construction in the case of an exponential nonlinearity).

Later, Merle and Zaag \cite{MZdm97} (see  also the note \cite{MZasp96}) simplified  the proof of \cite{BKnon94} and proved the stability of the constructed solution verifying the behavior \eqref{eq:1}. Their method relies on the linearization of the similarity variables version around the expected profile. In that setting, the linearized operator has two positive eigenvalues, then a non-negative spectrum. Then, they proceed in two steps:
\begin{itemize}
\item[-] Reduction of an infinite dimensional problem to finite dimensional one: they show that controlling the similarity variable version around the profile reduces to the control of the components corresponding to the two positive eigenvalues.
\item[-] Then, they solve the finite dimensional problem thanks to a topological argument based on index theory.
\end{itemize}

The method of Merle and Zaag \cite{MZdm97} has been proved to be successful in various situations. This was the case of the complex  Ginzgburg-Landau equation by Masmoudi and Zaag \cite{MZjfa08} (see also Zaag \cite{ZAAihn98} for an ealier work) and also for the  case of a complex semilinear heat equation with no variational structure by Nouaili and Zaag \cite{NZcpde15}. We also mention the work of Tayachi and Zaag \cite{TZpre15} (see also the note \cite{TZnor15}) and the work of Ghoul, Nguyen and Zaag \cite{GNZpre16a} dealing with a nonlinear heat equation with a double source depending on the solution and its gradient in a critical way. In \cite{GNZpre16c},  Ghoul, Nguyen and Zaag successfully adapted the method to construct a stable blowup solution for a non variational semilinear parabolic system.

In other words, the method of \cite{MZdm97} was proved to be efficient even for the case of systems with non variational structure. However, all the previous examples enjoy a common scaling invariant property like \eqref{eq:inv}, which seemed at first to be a strong requirement for the method.  In fact, this was proved to be untrue.

As matter of fact, Ebde and Zaag \cite{EZsema11} were able to adapt the method to construct blowup solutions for the following non scaling invariant equation
\begin{equation}\label{euqa-ebde-and-zaag}
\partial_t u = \Delta u + |u|^{p-1} u  + f(u,\nabla u),
\end{equation}
where
$$|f(u,\nabla u)| \leq C(1  + |u|^q +  |\nabla u|^{q'}), \text{ with } q < p, q' < \frac{2p}{p + 1}.$$
These conditions ensure that the perturbation $f(u, \nabla u)$ turns out to exponentially small coefficients in the similarity variables. Later, Nguyen and Zaag \cite{NZsns16} did a more spectacular achievement by addressing the case of stronger perturbation of \eqref{euqa-problem-specical-alpha=0}, namely 

\begin{equation}\label{equation-van-tien}
\partial_t u = \Delta u + |u|^{p-1} u + \frac{\mu |u|^{p-1} u}{\ln^a( 2 + u^2)},
\end{equation}
where $\mu \in \mathbb{R}$ and $a > 0$. When moving to the similarity variables, the perturbation turns out to have a polynomial decay. Hence, when $a > 0$ is small, we are almost in the case of a critical perturbation.

In both cases addressed in \cite{EZsema11} and \cite{NZsns16}, the equations are indeed non-scaling invariant, which shows the robustness of the method. However, since both papers proceed by perturbations around the standard case \eqref{euqa-problem-specical-alpha=0}, it is as if we are still in the scaling invariant case. 

In this paper, we aim at trying the approach on a genuinely non-scaling invariant case, namely equation \eqref{equ:problem}. This is our main result.

\begin{theorem}[Blowup solutions for equation \eqref{equ:problem} with a prescribed behavior]\label{existence} There exists an 
initial data $u_0 \in L^\infty(\mathbb{R}^n)$ such that the corresponding 
solution to equation \eqref{equ:problem} blows up in finite time $T=T(u_0) > 0,$
 only at the origin. Moreover, we have
\begin{itemize}
\item[$(i)$] For all $t \in [0,T)$, there exists a positive constant $C_0$ such that
\begin{equation}\label{estime-theorem}
\left\| \psi^{-1}(t) u(x,t)  - f_0\left(\frac{x}{\sqrt{(T-t)|\ln(T-t)|}}\right)  \right\|_{L^{\infty}(\mathbb{R}^n)} \leq \frac{C_0}{\sqrt {|\ln (T -t)|}},
\end{equation}
where $\psi(t)$ is the unique  positive solution of the following ODE
\begin{equation}\label{euqq-ODE-psi}
 \psi'(t)  = \psi^p(t) \ln^{\alpha}(\psi^2(t) +2), \quad \lim_{t \to T}\psi(t) = + \infty,
 \end{equation} 
 (see Lemma \ref{assymptotic-psi-T} for the existence and uniqueness of $\psi$), and the profile $f_0$ is defined by 
\begin{equation}\label{def:f0}
f_0(z) = \left( 1 + \frac{(p-1)}{4p}|z|^2 \right)^{-\frac{1}{p-1}}.
\end{equation} 
  
\item[$(ii)$] There exits $u^{*}(x) \in C^2(\mathbb{R}^n  \backslash \{0\})$ such that $u(x,t) \to u^{*}(x) \text{ as } t \to T$  uniformly on compact sets of $\mathbb{R}^n \setminus \{0\}$, where 
\begin{equation}\label{sharp-u-*}
u^*(x) \sim \left[ \frac{(p - 1)^2 |x|^2}{ 8 p |\ln|x||}\right]^{-\frac{1}{p  - 1}} \left( \frac{4 |\ln|x||}{p - 1} \right)^{-\frac{\alpha}{ p -  1}} \text{ as } x \to 0,
\end{equation}

\end{itemize}
\end{theorem}

\begin{remark} From $(i)$, we see that $u(0,t) \sim \psi(t) \to +\infty$ as $t \to T$, which means that the solution blows up in finite time $T$ at $x = 0$. From $(ii)$, we deduce that the solution blows up only at the origin.
\end{remark}

\begin{remark} Note that the behavior in \eqref{estime-theorem} is almost the same as in the standard case $\alpha = 0$ treated in \cite{BKnon94} and \cite{MZdm97}. However, the final profile $u^*$ has a difference coming  from the extra multiplication of the size $|\ln|x||^{-\frac{\alpha}{p-1}}$, which shows that the nonlinear source in equation \eqref{equ:problem} has a strong effect to the dynamic of the solution in comparing with the standard case $\alpha = 0$.
\end{remark}

\begin{remark}
Item  $(ii)$ is in fact  a consequence of  \eqref{estime-theorem} and Lemma \ref{citia-not-blow-up}. Therefore, the main goal of this paper is to construct for equation \eqref{equ:problem} a solution blowing up in finite time and verifying the behavior \eqref{estime-theorem}.
\end{remark}

\begin{remark} By the parabolic regularity, one can show that if the initial data $u_0 \in W^{2,\infty}(\mathbb{R}^n)$, then we have for $i = 0, 1, 2$,
$$\left\| \psi^{-1 }(t) (T  - t)^{ \frac{i}{2}} \nabla^i_x u (x,t)   - (T -t)^{\frac{i}{2}}\nabla^i_x f_0\left(\frac{x}{\sqrt{(T -t)|\ln(T-t)|}}\right) \right\|_{L^\infty} \leq \frac{C}{ \sqrt{ |\ln (T - t)|}},$$
where $f_0$ is defined by \eqref{def:f0}.
\end{remark}
From the technique of Merle \cite{Mercpam92}, we can prove the following result.
\begin{corollary}
For arbitrary given set of $m$ points $x_1,...,x_m$. There exists initial data $u_0$ such that the solution $u$ of \eqref{equ:problem} with initial data $u_0$ blows up exactly at  $m$ points $x_1,...,x_m$.  Moreover, the local behavior at each blowup point $x_i$ is also given by \eqref{estime-theorem}  by replacing $x$ by  $x - x_i$.
\end{corollary}
As a consequence of our technique, we prove the stability of the solution constructed in Theorem \ref{existence}  under the perturbations of initial data. In particular, we have the following  result.
\begin{theorem}[Stability of the  solution constructed in Theorem \ref{existence}]\label{stability}  Consider $\hat u$ the solution constructed in Theorem \ref{existence} and denote by $\hat T$ its blowup time. 
Then there exists  $\mathcal{U}_0 \subset L^\infty(\mathbb{R}^n)$ a neighborhood of $\hat u(0)$  such that for all  $ u_0 \in  \mathcal{U}_0$,  equation \eqref{equ:problem} with the initial data  $u_0$ has a  unique solution $u(t)$  blowing   up in finite time $T(u_0)$ at a single point $a(u_0)$. Moreover, the statements $(i)$ and $(ii)$ in Theorem \ref{existence} are satisfied  by $u(x-a(u_0),t)$, and
\begin{equation}\label{limit-u-0-hat-u}
\left( T(u_0), a(u_0)\right) \to (\hat T, 0)  \text{ as } \|u_0 - \hat u_0\|_{L^\infty(\mathbb{R}^n)} \to 0.
\end{equation}
\end{theorem}
\begin{remark} We will not give the proof of Theorem \ref{stability} because the stability result follows from the reduction to a finite-dimensional case as in \cite{MZdm97} with the same proof. Here we only prove the existence and refer to \cite{MZdm97} for the stability.
\end{remark}

\section{Formulation of the problem.}
In this section, we first use the matched asymptotic technique to formally derive the behavior \eqref{estime-theorem}. Then, we give the formulation of the problem in order to justify the formal result.
\subsection{ A formal approach.} \label{sec:formal}
 
In this part, we follow the approach of Tayachi and Zaag \cite{TZpre15} to formally explain how to derive the asymptotic behavior \eqref{estime-theorem}. To do so, we introduce the following self-similarity variables
\begin{equation}\label{similariy-variable}
u(x,t) = \psi(t) w(y,s),  \quad y = \frac{x}{ \sqrt{T -t}}, \quad  s  = -\ln (T -t),
\end{equation}
where $\psi(t)$ is the unique positive solution  of equation \eqref{euqq-ODE-psi} and $\psi(t) \to +\infty$ as $t \to T$. Then, we see from \eqref{equ:problem} that $w(y,s)$ solves the following equation: for all $(y,s) \in \mathbb{R}^n \times [-\ln T, +\infty)$ 
\begin{equation}\label{equation-w}
\partial_s w   = \Delta  w - \frac{1}{2} y. \nabla w - h(s) w  + h(s) |w|^{p-1}w \frac{\ln ^{\alpha}( \psi^2_1  w^2  + 2)}{ \ln^{\alpha}(\psi^2_1   +2)},
\end{equation} 
where
\begin{equation}\label{def-psi(s)}
h(s)  =  e^{-s} \psi^{p-1}_1(s) \ln^{\alpha}(\psi^2_1(s)  +2),
\end{equation}
and
\begin{equation}\label{defini-psi-1-in-s}
\psi_1(s) = \psi(T - e^{-s}).
\end{equation}
Note that $h(s)$ admits the following asymptotic behavior as $s \to +\infty$,
\begin{equation}\label{asym-pto-h-s}
h(s) =  \frac{1}{ p-1}  \left( 1 - \frac{\alpha}{ s}  - \frac{\alpha^2 \ln s}{s^2} \right) + O\left(\frac{1}{s^2}\right), 
\end{equation}
(see  ii) of Lemma \ref{function-h(s)} for the proof of \eqref{asym-pto-h-s}). From \eqref{similariy-variable}, we see that the study of the asymptotic behavior of $u(x,t)$ as $t \to T$ is equivalent to the study of the long time behavior of $w(y,s)$ as $s \to +\infty$. In other words, the construction of the solution $u(x,t)$, which blows up in finite time $T$ and verifies the behavior \eqref{estime-theorem},
reduces to the construction of a global solution $w(y,s)$ for equation \eqref{equation-w}  satisfying
\begin{equation}
0 < \epsilon_0 \leq \limsup_{s \to + \infty} \|w(s)\|_{L^{\infty}(\mathbb{R}^n)} \leq \frac{1}{\epsilon_0}, \epsilon_0 > 0,
\end{equation}
and 

\begin{equation}\label{estime-w-r-n}
\left\| w(y,s)  - \left( 1 + \frac{(p-1)y^2}{4p s} \right)^{-\frac{1}{p-1}}  \right\|_{L^{\infty}(\mathbb{R}^n)} \to 0   \quad \text{ as } s \to +\infty.
\end{equation}
In the following, we will formally explain how to derive the behavior \eqref{estime-w-r-n}.

\subsubsection{Inner expansion}
We remark that  $ 0 , \pm 1$ are the trivial  constant  solutions  to equation \eqref{equation-w}. Since  we  are looking  for a non zero solution, let us  consider the case when  $w  \to 1 $ as  $s \to +\infty$. We now  introduce 
\begin{equation}\label{w=barw+1}
w = 1  + \bar w,
\end{equation}
then from equation \eqref{equation-w},  we see that $\bar w$ satisfies

\begin{equation}\label{euqabarw}
\partial_s \bar w = \mathcal{L}(\bar w) + N(\bar w,s),
\end{equation}
where 
\begin{align}
\mathcal{L}  & =  \Delta   - \frac{1}{2}  y . \nabla   + \textup{Id},\label{def-ope-macalL}\\
N(\bar w, s) & =  \displaystyle  h(s) |\bar w   +1|^{p-1}(\bar w  +1) \frac{\ln^{\alpha}(\psi^2_1 (\bar w  +1)^2 +2)}{ \ln^{\alpha}( \psi^2_1 +2)} - h(s)(\bar w  +1) - \bar w \label{def:N-2},
\end{align}
$\psi_1(s)$ is defined in \eqref{defini-psi-1-in-s}  and $h(s)$ behaves as  in \eqref{asym-pto-h-s}.  Note that $N $ admits the following asymptotic behavior,
\begin{equation}\label{intro-result-asym-N_1}
 N(\bar w, s) =  \frac{ p \bar w^2}{2} + O\left(\frac{|\bar w|\ln s }{s^2}\right) + O\left(\frac{|\bar w|^2}{s}\right) + O(|\bar w|^3) \quad  \text{ as } \quad (\bar w ,s)  \to (0 ,+\infty),
 \end{equation}
(see Lemma \ref{asymptotic-N-1-N-2} for the proof of this statement). 

Since $\bar w(s) \to 0$ as $s \to +\infty$ and the nonlinear term $N$ is quadratic in $\bar w$, we see from equation \eqref{euqabarw} that the linear part will play the main role in the analysis of our solution. Let us recall some properties of $\mathcal{L}$. 
The linear operator $\mathcal{L}$ is self-adjoint in $L^2_\rho(\mathbb{R}^n)$, where $L^2_\rho$ is the weighted space associated with the weight $\rho$ defined by
$$\rho(y)   =\frac{e^{- \frac{|y|^2}{4}}}{(4 \pi)^{\frac{n}{2}}},$$
and 
$$\textup{spec}(\mathcal{L}) = \left\{1 - \frac m2, m \in \mathbb{N}\right\}.$$
More precisely, we have
\begin{itemize}
\item When $n = 1$, all the eigenvalues of $\mathcal{L}$ are simple and the eigenfunction corresponding to the eigenvalue $1 - \frac m2$ is the Hermite polynomial defined by
\begin{equation}\label{Hermite}
h_m(y)   = \sum_{j=0}^{\left[ \frac{m}{2}\right]} \frac{(-1)^jm! y^{m - 2j}}{j! ( m -2j)!}.
\end{equation}
In particular, we have the following orthogonality
 $$\int_{\mathbb{R}} h_i h_j  \rho dy  =  i! 2^i \delta_{i,j}, \quad \forall (i,j) \in \mathbb{N}^2. $$

\item When $n \geq 2$, the eigenspace corresponding to the eigenvalue $1 - \frac {m}{2}$ is defined as follows
\begin{equation}\label{eigenspace}
 \mathcal{E}_m   = \left\{ h_{\beta} = h_{\beta_1} \cdots h_{\beta_n}, \text{ for all } \beta \in \mathbb{N}^n, |\beta|  = m , |\beta| = \beta_1 + \cdots +\beta_n  \right\}.
\end{equation}
\end{itemize}
Since the eigenfunctions of $\mathcal{L}$  is a basic of $L^2_\rho$, we can expand $\bar w$ in this basic as follows
$$ \bar w(y,s)  = \sum_{\beta \in \mathbb{N}^n} \bar w_{\beta}(s)h_{\beta}(y) .$$
For simplicity, let us assume that $\bar w$ is radially symmetric in $y$. 
Since $h_\beta$ with $|\beta| \geq 3$ corresponds to negative eigenvalues of $\mathcal{L}$, we may consider the solution $\bar w$ taking the form
\begin{equation}\label{form-barw}
\bar w =  \bar w_0  + \bar w_2(s)(|y|^2  - 2n),
\end{equation} 
where $|\bar w_0(s)|$ and $|\bar w_2(s)|$ go to $0$ as $s \to +\infty$. Injecting \eqref{form-barw} and  \eqref{intro-result-asym-N_1} into \eqref{euqabarw}, then  projecting equation \eqref{euqabarw} on 
the  eigenspace $\mathcal{E}_m $ with $m=0 $ and $m =2,$
\begin{equation}
\left\{ \begin{array}{l}
\bar w_0' = \displaystyle \bar w_0  + \frac{p}{2} \left( \bar w_0^2 + 8 n \bar w_2^2  \right) + O\left(\frac{ ( |\bar w_0|  +|\bar w_2| )\ln s}{s^2} \right)    \\[0.3cm]
\quad \quad \quad +\displaystyle O\left( \frac{|\bar w_0|^2 + |\bar w_2|^2}{s} \right) + O\left(|\bar w_0|^3  + |\bar w_2|^3\right),\\[0.3cm]
\bar w_2'  = \displaystyle 4p \bar w_2^2 + p \bar w_0 \bar w_2  + O\left(\frac{ (|\bar w_0|  +|\bar w_2| )\ln s}{s^2}\right)\\[0.3cm]
\quad \quad \quad +\displaystyle O\left( \frac{|\bar w_0|^2 + |\bar w_2|^2}{s}\right) + O\left(|\bar w_0|^3  + |\bar w_2|^3\right),
\end{array}
\right.
\end{equation}
as $s \to + \infty$. we now   assume   that  $|\bar w_0(s)| \ll |\bar w_2(s)|$ as $s \to +\infty$, then \eqref{ode-system} becomes
\begin{equation}\label{ode-system}
\left\{ \begin{array}{l}
\bar w_0' = \bar w_0  + O(|\bar w_2|^2)   +O\left(\frac{|\bar w_2| \ln s}{s^2}\right),\\[0.3cm]
\bar w_2' =   4p \bar w_2^2  + o(|\bar w_2|^2)   +O\left(\frac{|\bar w_2| \ln s}{s^2} \right),
\end{array} \text{ as } s \to + \infty.
\right.
\end{equation}
We consider the following cases:

- Case 1: Either  $|\bar w_2| = O\left( \frac{ \ln s}{s^2}\right)$ or $|\bar w_2| \ll \frac{\ln s}{s}$  as $s \to + \infty$, then the second equation in \eqref{ode-system} becomes 
$$\bar w_2' = O\left( \frac{|\bar w_2| \ln s}{s^2}\right) \text{ as } s  \to  +\infty,$$
which yields
$$\ln  |\bar w_2|  = O\left( \frac{\ln s}{s}\right) \text{ as } s \to +\infty,$$
which  contradicts with the condition $\bar w_2(s) \to 0 $ as $s \to +\infty$.

- Case 2: $|\bar w_2| \gg \frac{\ln s}{s^2}$ as $s \to + \infty$, then \eqref{ode-system} becomes
$$\left\{ \begin{array}{l}
\bar w_0' = \bar w_0  + O(|\bar w_2|^2) ,\\[0.3cm]
\bar w_2' =   4p \bar w_2^2  + o(|\bar w_2|^2),
\end{array} \text{ as } s \to + \infty.
\right.$$ 
This yields 
\begin{equation}\label{ru-sult-ODE-sys}
\left\{ \begin{array}{l}
\bar w_0 = O\left( \frac{1}{s^2}\right) ,\\[0.3cm]
\bar w_2 =    - \frac{1}{4ps}  + o(\frac{1}{s}),
\end{array} \text{ as } s \to + \infty.
\right.
\end{equation}
Substituting  \eqref{ru-sult-ODE-sys} into \eqref{ode-system}   yields
$$
\left\{ \begin{array}{l}
\bar w_0' = O\left( \frac{1}{s^2}\right) ,\\[0.3cm]
\bar w_2' =  4p \bar w^2_2  +  O\left(\frac{\ln s}{s^3}\right),
\end{array} \text{ as } s \to + \infty,
\right.$$
from which we improve the error for $ \bar w_2$ as follows
\begin{equation}\label{system-ODE-barw-0-1}
\left\{ \begin{array}{l}
\bar w_0 =\displaystyle  O\left(\frac{1}{s^2}\right),\\[0.3cm]
\bar w_2   = \displaystyle - \frac{1}{4p s} + O\left( \frac{\ln ^2s}{s^2} \right),
\end{array} \text{ as } s \to +\infty . 
\right.
\end{equation}
Hence, from \eqref{w=barw+1}, \eqref{form-barw} and \eqref{system-ODE-barw-0-1}, we derive 
\begin{equation}\label{assym-w}
w(y,s)    = 1  - \frac{y^2}{4 ps }  + \frac{n}{2ps}  + O\left(  \frac{\ln^2s}{s^2}\right),
\end{equation}
in $L^2_\rho (\mathbb{R}^n)$ as $s \to + \infty$. Note that the asymptotic expansion \eqref{assym-w} also holds for all $|y| \leq K$, K is an arbitrary positive number.
\subsubsection{Outer expansion.}
The  asymptotic behavior  of \eqref{assym-w} suggests that the blowup profile depends on the variable 
$$z = \frac{y}{\sqrt s}, $$
From  \eqref{assym-w}, let us try to  search  a regular solution of equation \eqref{equation-w} of the form
\begin{equation}\label{assy-profile}
w(y,s) = \phi_0(z) + \frac{n}{2 ps} + o\left( \frac{1}{s}\right) \text{ in } L^\infty_{loc} \text{ as } s \to +\infty,
\end{equation}
where $\phi_0$ is a  bounded, smooth function to be determined. From \eqref{assym-w}, we impose the condition 
\begin{equation}\label{condi-phi-0}
\phi_0(0) = 1.
\end{equation}
Since $w(y,s)$ is supposed to be bounded, we obtain from Lemma \ref{asymptotic-zp-h-s-ln-al} that 
$$\left| h(s) | w|^{p-1} w \frac{\ln^{\alpha}(\psi_1^2 w^2  +2)}{\ln^{\alpha}(\psi_1^2  +2)} -  \frac{|w|^{p-1}w}{p-1} \right| =  O\left(\frac{1}{s}\right),$$
Note also that
$$\left| \left|\phi_0(z) +  O\left( \frac{1}{s}\right)\right|^{p-1} \left(\phi_0(z) +  O\left( \frac{1}{s}\right)\right)   - |\phi_0(z)|^{p-1} \phi_0(z)  \right|  =  O\left(\frac{1}{s}\right).$$
Hence, injecting \eqref{assy-profile} into equation \eqref{equation-w} and comparing terms of order $O\left( \frac{1}{s^i}\right)$ for $j = 0, 1, \cdots$, we derive the following equation for $j = 0$,
\begin{equation}\label{equa:phi}
- \frac{1}{2} z .\nabla \phi_0(z) - \frac{\phi_0(z)}{p-1} + \frac{|\phi_0|^{p-1}\phi_0(z)}{p-1} = 0, \quad  \forall z \in \mathbb{R}^n.
\end{equation}
Solving \eqref{equa:phi} with condition \eqref{condi-phi-0}, we obtain  
\begin{equation}\label{function-phi}
\phi_0(z)  = \left( 1  + c_0 |z|^2 \right)^{-\frac{1}{p-1}},
\end{equation}
for some constant  $c_0 \geq 0$ (since we want $\phi_0$ to be bounded for all $z \in \mathbb{R}^n$). From \eqref{assy-profile}, \eqref{function-phi} and a Taylor expansion, we obtain
$$w(y,s) = 1 - \frac{c_0 y^2}{(p-1)s} + \frac{n}{2ps}  + o\left(\frac{1}{s} \right), \quad \forall |y| \leq K \text{ as } s \to + \infty ,$$
from which and the asymptotic behavior \eqref{assym-w}, we find that   
$$ c_0 = \frac{p-1}{4p}.$$
In conclusion,  we have  just derived  the following asymptotic profile  
\begin{equation}\label{eq:formexp}
 w(y,s) \sim \varphi(y,s) \quad  \text{ as } \;\; s \to +\infty,
 \end{equation}
where
\begin{equation}\label{def-varphi}
\varphi(y,s) = \left( 1+ \frac{(p-1)y^2}{4ps}\right)^{-\frac{1}{p-1}}  + \frac{n}{2ps}.
\end{equation}
\subsection{Formulation of the problem.}
In this subsection, we set up the problem  in order to  justify  the formal approach  presented  in the Section \ref{sec:formal}. In particular, we give a formulation to prove item $(i)$ of Theorem \ref{existence}. We aim at constructing for equation \eqref{equ:problem} a solution blowing up in finite time $T$ only at the origin and verifying the behavior \eqref{estime-theorem}. In the similarity variables \eqref{similariy-variable}, this is equivalent to the construction of a solution $w(y,s)$ for equation \eqref{equation-w}  defined for all $(y,s) \in \mathbb{R}^n \times [s_0, +\infty)$ and satisfying \eqref{estime-w-r-n}. The formal approach given in subsection \ref{sec:formal} (see \eqref{eq:formexp}) suggests to linearize $w$ around the profile function $\varphi$ defined by \eqref{def-varphi}. Let us introduce
\begin{equation}\label{def:q}
q(y,s)  = w(y,s) - \varphi(y,s),
\end{equation}
where $\varphi$ is defined by \eqref{def-varphi}. From \eqref{equation-w}, we see that $q$ satisfies the equation
\begin{equation}\label{equa-q}
\partial_s q   = \mathcal{L} q  + Vq  + B(q)  + R(y,s) + D(q,s),
\end{equation}
where $\mathcal{L}$ is the linear operator defined by  \eqref{def-ope-macalL} and 
\begin{align}
V &= \frac{p}{p-1} \left[ \varphi^{p-1}  -1\right], \label{def:V}\\[0.3cm]
B(q)  &= \frac{|q + \varphi|^{p-1} (q + \varphi) - \varphi^p - p \varphi^{p-1} q}{p-1},\label{def:B-q}\\[0.3cm]
R(y,s)  &= \Delta \varphi - \frac{1}{2} y \nabla \varphi  - \frac{\varphi}{p-1}  + \frac{\varphi^p}{p-1} - \partial_s \varphi,\label{def:R}\\[0.3cm]
D(q,s)  &= (q + \varphi)\left( \left( h(s)  -\frac{1}{p-1}\right) \left( |q + \varphi|^{p-1} - 1 \right)  + h(s) |q + \varphi|^{p-1} (q + \varphi) L(q + \varphi, s)\right)\label{def:D},\\
L(v, s) & = \frac{2 \alpha  \psi^2_1}{\ln(\psi_1^2  + 2)(\psi_1^2 +2)}(v -1)  + \frac{1}{\ln^{\alpha}(\psi_1^2 +2)}\int_{1}^v  f''(u) (v -u)du\label{def:L-first},
\end{align}
with $h, \psi_1(s)$ and  $\varphi$ being defined by  \eqref{def-psi(s)}, \eqref{defini-psi-1-in-s} and  \eqref{def-varphi}  respectively, and
$$f(z)  = \ln^{\alpha}(\psi_1^2 z^2  + 2), z \in \mathbb{R}.$$
Hence, proving \eqref{estime-theorem} now reduces to construct for equation \eqref{equa-q} a solution $q$ such that 
$$\lim_{s \to +\infty}\|q(s)\|_{L^\infty} \to 0.$$
Since we construct for equation \eqref{equa-q} a solution $q$ verifying $\|q(s)\|_{L^\infty} \to 0”$ as $s \to +\infty$, and the fact that
$$|B(q)| \leq C|q|^{\min{(2,p)}}, \quad \|R(s)\|_{L^\infty}  +  \|D(q,s)\|_{L^\infty} \leq \frac Cs,$$
(see Lemmas \ref{inside-D}, \ref{lemma-of-rest-termes} and \ref{estimate-B-q} for these estimates), we see that the linear part of equation \eqref{equa-q} will play an important role in the analysis of the solution.
The property of the linear operator $\mathcal{L}$ has been studied in previous section (see page \pageref{Hermite}),  and the potential $V$ has the following properties:

$i)$ Perturbation of effect of $\mathcal{L}$ inside the blowup region $\{|y| \leq K\sqrt s\}$:
$$\|V(s)\|_{L^2_\rho} \to 0 \quad \text{as} \;\; s\to +\infty.$$

$ii)$ For each  $\epsilon > 0$, there exist $K_{\epsilon} >0$ and $ s_{\epsilon} >0$ such that
$$ \sup_{\frac{y}{\sqrt s} \geq K_{\epsilon}, s \geq s_{\epsilon}} \left| V(y,s)  + \frac{p}{p-1} \right| \leq    \epsilon.$$ 
Since $1$ is the biggest eigenvalue of $\mathcal{L}$, the operator $\mathcal{L}+ V$  behaves as one with with a fully negative spectrum  outside blowup region $\{|y| \geq K\sqrt s\}$, which makes the control of the solution in this region easily.

Since the behavior of the potential $V$ inside and outside the blowup region is different,
 we will consider the dynamics of the solution for   $|y| \leq 2K\sqrt s$ and for $|y| \geq K\sqrt s$ separately for some K to be fixed large.
We introduce the following function
\begin{equation}\label{def-chi}
\chi(y,s)  = \chi_0\left(\frac{|y|}{K \sqrt s} \right),
\end{equation}
where $\chi_0 \in C^{\infty}_{0}[0,+\infty), \|\chi_0\|_{L^{\infty}} \leq 1$ and 
$$
\chi_0(x) = \left\{  \begin{array}{l}
 1 \quad \text{ for } x  \leq 1,\\
 0 \quad   \text{ for }  x  \geq 2,
\end{array} \right.$$
and  $K$ is a positive constant to be fixed large later. We now decompose $q$ by
\begin{equation}\label{decomp-q1}
q = \chi q + (1  - \chi) q =  q_b  + q_e.
\end{equation} 
(Note that $\textup{supp} (q_b) \subset \{|y| \leq 2 K \sqrt s\}$ and $\textup{supp} (q_e) \subset \{|y| \geq  K \sqrt s\}$). Since the eigenfunctions of $\mathcal{L}$ span the whole space $L^2_\rho$, let us write
\begin{equation}\label{decomp-q-b}
q_b(y,s)  =  q_0(s) + q_1(s) \cdot y + \frac{1}{2} y^T \cdot q_2(s) \cdot y - \textup{tr}(q_2(s))+ q_-(y,s),
\end{equation}
where $q_m(s) = \big(q_\beta(s)\big)_{\beta \in \mathbb{N}^n, |\beta| = m}$ and
\begin{equation}
\forall \beta \in \mathbb{N}^n, \quad q_\beta(s) = \int_{\mathbb{R}^n}  q_b(y,s) \tilde h_\beta(y) \rho dy, \quad  \tilde h_\beta =  \frac{h_\beta}{\|h_\beta\|^2_{L^2_\beta}},
\end{equation}
and 
\begin{equation}
q_-(y,s) = \sum_{\beta \in \mathbb{N}^n, |\beta| \geq 3} q_\beta(s) h_\beta(y).
\end{equation}
In particular, we denote  $q_1  = (q_{1,i})_{1 \leq i \leq n}$ and $q_2(s)$ is a $n \times n$ symmetric matrix defined explicitly by
\begin{equation}\label{defini-of_q_2}
 q_2(s)  = \int q_b \mathcal{M}(y) \rho dy = (q_{2,i,j})_{1 \leq i,j \leq n},
\end{equation}
with 
\begin{equation}\label{def-matrix-h-2} 
\mathcal{M} = \left\{ \frac{1}{4} y_i y_j  - \frac{\delta_{i,j}}{2}\right\}_{1 \leq i,j \leq n}.
\end{equation}
Hence, by \eqref{decomp-q1} and \eqref{decomp-q-b}, we can write
\begin{equation}\label{decomp-q2}
q(y,s)  = q_0(s) + q_1(s) \cdot y + \frac{1}{2} y^T \cdot q_2(s) \cdot y - \textup{tr}(q_2(s))  + q_-(y,s) + q_e(y,s).
\end{equation}
Note that $q_m(m=0,1,2)$ and $q_-$ are the components of $q_b$, and not those of $q$. 

\section{Proof of the existence assuming some technical results.}
In this section, we shall describe the main argument behind the proof of Theorem \ref{existence}. To avoid winding up with details, we shall postpone most of the technicalities involved to the next section. 
According to the transformations \eqref{similariy-variable} and \eqref{def:q}, proving $(i)$ of Theorem \ref{existence} is equivalent to showing that there exists an initial data  $q_0(y)$ at the time $s_0$  such that the corresponding solution $q$ of  equation \eqref{equa-q} satisfies
$$ \|q(s)\|_{L^{\infty}(\mathbb{R}^n)} \to  0 \text{ as } s \to + \infty.$$
In particular, we consider the following function
\begin{equation}\label{def-psi0}
\psi_{d_0,d_1}(y) = \frac{A}{s_0^2} \left( d_0  + d_1 .y\right)\chi (2 y , s_0),
\end{equation}
as the initial data for equation \eqref{equa-q}, where $(d_0, d_1) \in \mathbb{R}^{1 + n}$ are the parameters to be determined, $s_0 > 1$ and $A > 1$ are constants to be fixed large enough, and  $\chi $ is the function defined by \eqref{def-chi}. 

We aim at proving that there exists $(d_0,d_1) \in \mathbb{R} \times \mathbb{R}^n$ such that the solution $q(y,s) = q_{d_0, d_1}(y,s)$ of \eqref{equa-q} with  initial data $\psi_{d_0,d_1}(y)$  satisfies
$$\|q_{d_0,d_1}(s)\|_{L^{\infty}}   \to  0 \text{ as } s \to + \infty.$$
More precisely, we will show that there exists $(d_0,d_1) \in \mathbb{R} \times \mathbb{R}^n$ such that the solution $q_{d_0,d_1}(y,s)$ belongs to the shrinking set $S_A$ defined as follows:
\begin{definition}[A shrinking set to zero]\label{def-V(s)1}
For all $A \geq 1, s \geq 1$ we define $S_A(s)$ being the set of all functions  $q \in L^{\infty}(\mathbb{R}^n)$ such that
\begin{align*}
&|q_0| \leq  \frac{A}{s^2}, \quad |q_ {1,i}| \leq  \frac{A}{s^2}, \quad |q_{2, i,j}| \leq   \frac{ A^2 \ln^2 s}{s^2}, \quad \forall  1 \leq i,j \leq n,\\
&\left\| \frac{q_-(y)}{1  +|y|^3}\right\|_{L^{\infty}(\mathbb{R}^n)} \leq  \frac{A}{s^2}, \quad \|q_e(y)\|_{L^{\infty}(\mathbb{R}^n)} \leq  \frac{A^2}{\sqrt s},
\end{align*}
where $q_0$, $q_1 = \big(q_{1,i}\big)_{1\leq i \leq n}$, $q_2 = \big(q_{2,i,j} \big)_{1 \leq i,j\leq n}$, $q_-$ and $q_e$ are defined as in \eqref{decomp-q2}.

We also denote by $\hat S_A(s)$ being the set 
\begin{equation}\label{def:SAhat}
\hat S_A(s) = \left[ - \frac{A}{s^2},\frac{A}{ s^2}\right] \times \left[ - \frac{A}{s^2},\frac{A}{ s^2}\right]^n.
\end{equation}
\end{definition}
\begin{remark}\label{remark-inequa-q-poli} For each $ A \geq 1, s \geq 1$, we have the following estimates for all $q(s) \in S_A(s)$:
\begin{equation}\label{remark-on-q-b}
 |q(y,s)| \leq \frac{C A^2 \ln^2 s }{s^2}(1 + |y|^3), \quad \forall y \in\mathbb{R}^n,
\end{equation}
\begin{equation}\label{norm-q-inside}
\|q(s)\|_{L^{\infty}(\{|y| \leq 2 K \sqrt s\})} \leq  \frac{CA}{\sqrt s},
\end{equation}
\begin{equation}\label{norm-q-allspace}
\|q(s)\|_{L^{\infty}(\mathbb{R}^n)} \leq  \frac{CA^2}{\sqrt s}.
\end{equation}
\end{remark}
We aim at proving the following central proposition which implies Theorem \ref{existence}.
\begin{proposition}[Existence of a solution trapped in $S_A(s)$]\label{pro-exist}
There exists $A_1 \geq 1 $ such that for all $ A \geq A_1$ there exists $s_1(A) \geq 1$ such that for all $s_0 \geq s_1(A)$, there exists $(d_0,d_1) \in \mathbb{R}^{1 + n}$ such that the solution $q(y,s) = q_{d_0,d_1}(y,s)$ of \eqref{equa-q} with the initial data at the time $s_0$ given by $q(y,s_0) = \psi_{d_0,d_1}(y)$, where $\psi_{d_0,d_1}$ is defined as in \eqref{def-psi0}, satisfies
$$q(s)  \in S_A(s), \quad \forall s \in [s_0,+\infty).$$
\end{proposition}
From \eqref{norm-q-allspace}, we  see that once Proposition \ref{pro-exist} is proved, item $(i)$ of Theorem \ref{existence} directly follows. In the following, we shall give all the main arguments for the proof of this proposition assuming some technical results which are left to the next section. 

As for the initial data at time $s_0$ defined as in \eqref{def-psi0}, we have the following properties.
\begin{proposition}[Properties of the initial data \eqref{def-psi0}]\label{pro-initial}
For each $A \geq 1$,  there exists $s_2(A) > 1 $ such that for all $s_0 \geq s_2(A)$ we have the following properties:

\begin{itemize}
\item[$i)$] There exists $\mathbb{D}_{A,s_0} \subset [-2;2] \times [-2;2]^n$  such that the mapping 
\begin{eqnarray*}
\Phi_1 : \mathbb{R}^{1 + n} & \to & \mathbb{R}^{1+n},\\
(d_0,d_1) & \mapsto & \big(\psi_0, \psi_{1}\big)
\end{eqnarray*}
 is linear, one to one from $\mathbb{D}_{A,s_0}$ onto $\hat S_A(s_0)$. Moreover 
$$\Phi_1 \left( \partial \mathbb{D}_{A, s_0} \right) \subset \partial \hat S_A(s_0).$$ 
\item[$ii)$] For all $(d_0,d_1) \in \mathbb{D}_{A, s_0}$ we have  $\psi_{d_0,d_1} \in S_A(s_0)$ with strict inequalities in the sense that
\begin{align*}
&|\psi_0|  \leq \frac{A}{s_0^2}, \quad |\psi_{1,i}| \leq \frac{A}{s_0^2}, \quad |\psi_{2,i,j}|  <  \frac{A \ln^2 s_0}{s^2_0}, \quad \forall 1 \leq i,j \leq n, \\
& \left\| \frac{\psi_-}{1 + |y|^3} \right\|_{L^{\infty}(\mathbb{R})}  <  \frac{A}{s^2_0}, \quad \psi_e \equiv 0.
\end{align*}
\end{itemize}
where $\chi(y,s_0)$ is defined in  \eqref{def-chi}, $\psi_0,   (\psi_{1,i})_{1 \leq i \leq n}, (\psi_{2,i,j})_{1 \leq i,j \leq 2}, \psi_-$, $\psi_e$ are the components of $\psi_{d_0,d_1}$ defined as in \eqref{decomp-q2}, $\psi_{d_0,d_1}$ and $\hat S_A(s)$ are defined by \eqref{def-psi0} and \eqref{def:SAhat}.
\end{proposition}
\begin{proof} See Propositon 4.5 of Tayachi and Zaag  \cite{TZpre15} for a similar proof of this proposition. \end{proof}

From now on, we denote by $C$ the universal constant which only depends on $K$, where $K$ is introduced in \eqref{def-chi}. Let us now give the proof of Proposition \ref{pro-exist} to complete the proof of item $(i)$ of Theorem \ref{existence}. 
\begin{proof}[\textbf{Proof of Proposition \ref{pro-exist}}.] We proceed into two steps to prove Proposition \ref{pro-exist}:\\
- In the first step, we reduce the problem of controlling $q(s)$ in $S_A(s)$ to the control of $(q_0,q_1)(s)$ in $\hat S_A(s)$, where $q_0$ and $q_1$ are the component of $q$ corresponding to the positive modes defined as in \eqref{decomp-q2} and $\hat S_A$ is defined by \eqref{def:SAhat}. This means that we reduce the problem to a finite dimensional one.\\
- In the second step, we argue by contradiction to solve the finite dimensional problem thanks to a topological argument. \\

\noindent \textit{Step 1: Reduction to a finite dimensional problem.}

In this step, we show through \textit{a priori estimate} that the control of $q(s)$ in $S_A(s)$ reduces to the control of $(q_0,q_1)(s)$ in $\hat S_A(s)$. This mainly follows from a good understanding of the properties of the linear part $\mathcal{L} + V$ of equation \eqref{equa-q}. In particular, we claim the following which is the heart of our analysis.
\begin{proposition}[Control of $q(s)$ in $S_A(s)$ by $(q_0,q_1)(s)$ in $\hat S_A(s)$] \label{prop:redu} There exists $A_3 \geq 1$ such that for all $A \geq A_3$, there exists $s_3(A) \geq 1$ such that for all $s_0 \geq s_3(A)$, the following holds:\\
If $q(y,s)$ is the solution of equation \eqref{equa-q} with the initial data at time $s_0$ given by \eqref{def-psi0} with $(d_0,d_1) \in \mathbb{D}_{A,s_0}$, and $q(s) \in S_A(s)$ for all $s \in [s_0, s_1]$ for some $s_1 \geq s_0$ and $q(s_1) \in \partial S_A(s_1)$, then:\\
$(i)\;$ (Reduction to a finite dimensional problem) We have $(q_0,q_1)(s_1) \in \partial \hat S_A(s_1)$.\\
$(ii)\,$ (Transverse outgoing crossing) There exists $\delta_0 > 0$ such that
$$\forall \delta \in (0, \delta_0), \quad (q_0, q_1)(s_1 + \delta) \not \in \hat S_A(s_1 + \delta),$$
hence, $q(s_1 + \delta)\not \in S_A(s_1 + \delta)$, where  $\hat S_A$ is defined in  \eqref{def:SAhat} and $\mathbb{D}_{A,s_0}$ is introduced in Proposition   \ref{pro-initial}.
\end{proposition}
Let us suppose for the moment that Proposition \ref{prop:redu} holds. Then we can take advantage of a topological argument quite similar to that already used in \cite{MZdm97}. \\

\noindent \textit{Step 2: A basic topological argument.}

From Proposition \ref{prop:redu}, we claim that there exists $(d_0,d_1) \in \mathbb{D}_{A,s_0}$ such that equation \eqref{equa-q} with initial data \eqref{def-psi0} has a solution 
$$q_{d_0,d_1}(s) \in S_A(s), \quad \forall s \in [s_0, +\infty),$$
for suitable choice of the parameters $A, K, s_0$. Since the argument is analogous as in \cite{MZdm97}, we only give the main ideas. 

Let us consider $s_0, K, A$ such that Propositions \ref{pro-initial}  and \ref{prop:redu}  hold. From Proposition \ref{pro-initial}, we have 
$$\forall (d_0, d_1) \in \mathbb{D}_{A,s_0}, \quad q_{d_0,d_1}(y,s_0):=\psi_{d_0,d_1} \in S_A(s_0),$$
where $\psi_{d_0,d_1}$ is defined by \eqref{def-psi0}. Since the initial data belongs to $L^\infty$, we then deduce from the local existence theory for the Cauchy problem of \eqref{equ:problem} in $L^\infty$ that we can define for each $(d_0,d_1) \in \mathbb{D}_{A,s_0}$ a maximum time $s_*(d_0,d_1) \in [s_0, +\infty)$ such that 
$$q_{d_0,d_1}(s) \in S_A(s), \quad \forall s \in [s_0, s_*).$$
If $s_*(d_0,d_1) = +\infty$ for some $(d_0, d_1) \in \mathbb{D}_{A,s_0}$, then we are done. Otherwise, we argue by contradiction and assume that $s_*(d_0, d_1) < +\infty$ for all $(d_0, d_1) \in \mathbb{D}_{A,s_0}$. By continuity and the definition of $s_*$, we deduce that $q_{d_0,d_1}(s_*)$ is on the boundary of $S_A(s_*)$. From item $(i)$ of Proposition \ref{prop:redu}, we have 
$$(q_0,q_1)(s_*) \in \partial \hat S_A(s_*).$$
Hence, we may define the rescaled function
\begin{align*}
\Gamma: \; \mathbb{D}_{A,s_0} &\mapsto \partial \big([-1,1]^{1 + n}\big)\\
 (d_0,d_1) &\to \frac{s_*^2}{A}(q_0, q_1)(s_*).
\end{align*} 
From item $(i)$ of Proposition \ref{pro-initial}, we see that if $(d_0,d_1) \in \partial \mathbb{D}_{A,s_0}$, then 
$$q(s_0) \in S_A(s_0), \quad (q_0,q_1)(s_0) \in \partial \hat S_A(s_0).$$
From item $(ii)$ of Proposition \ref{prop:redu}, 
we see that $q(s)$ must leave $S_A(s)$ at $s = s_0$, 
thus, $s_*(d_0,d_1) = s_0$. Therefore, the restriction of  $\Gamma$to $\partial \mathbb{D}_{A,s_0}$ is  homeomorphic  to the  identity mapping, which is impossible thanks to index theorem, and the contradiction is obtained. This concludes the proof of Proposition \ref{pro-exist} as well as item $(i)$ of Theorem \ref{existence}, assuming that Proposition \ref{prop:redu} holds.
\end{proof}

We now give the proof of item $(ii)$ of Theorem \ref{existence}.
\begin{proof}[Proof of item $(ii)$ of Theorem \ref{existence}] The existence of $u^*$ in $C^2(\mathbb{R}^n \setminus \{0\})$ follows from the technique of Merle \cite{FM1992}. Here, we want to  find an equivalent formation for $u^*$ near  the  blowup point $x = 0$. The case $\alpha = 0$ was treated in \cite{ZAAihn98}. When $\alpha \neq 0$, we follow the method of \cite{ZAAihn98}, and no new idea is needed. Therefore, we just sketch the main steps for the sake of completeness.

We consider $K_0 > 0$ the constant to be fixed large enough, and $|x_0| \neq  0 $ small enough. Then, we introduce the following function 
\begin{equation}\label{equa-upsilon-xi-tau}
\upsilon (x_0, \xi, \tau)  = \psi^{-1} ( t_0(x_0))u(x,t), 
\end{equation}
where $(\xi, \tau) \in \mathbb{R}^n \times \left[ - \frac{t_0(x_0)}{T - t_0(x_0)}, 1 \right)$, and 
\begin{equation}\label{relation-x-and-xi-tau}
(x,t) = \big(x_0 + \xi\sqrt{T - t_0(x_0)},t_0(x_0) + \tau (T - t_0(x_0))\big),
\end{equation}
with $t_0(x_0) $  being uniquely determined  by 
\begin{equation}\label{equa-x-0-and-t-0-x-0}
|x_0| = K_0 \sqrt{(T - t_0(x_0)) |\ln(T -  t_0(x_0))|}.
\end{equation}
From \eqref{equa-upsilon-xi-tau}, \eqref{relation-x-and-xi-tau} , \eqref{equa-x-0-and-t-0-x-0} and \eqref{estime-theorem} we derive that
$$ \sup_{|\xi| < 2 |\ln(T - t_0(x_0))|^{\frac{1}{4}}} \left| v (x_0, \xi, 0) - \varphi_0(K_0)\right| \leq \frac{C}{ 1 + (|\ln(T - t_0(x_0))|^{\frac{1}{4}})} \to 0 \quad \text{ as } x_0 \to 0,$$
where $\varphi_0 (x)= \left( 1 + \frac{(p - 1)x^2}{4p} \right)^{\frac{1}{p-1}}$. As in \cite{ZAAihn98}, we use the continuity with respect to initial data for equation \eqref{equ:problem} associated to a space-localization in the ball $B(0, |\xi| < |\ln(T - t_0(x_0))|^{\frac{1}{4}})$ to derive
\begin{equation}\label{sup-v-xi-tau-apro-1}
\sup_{|\xi| <  |\ln(T - t_0(x_0))|^{\frac{1}{4}}, \tau \in [0,1)} \left| v (x_0, \xi, \tau) - \hat v_{K_0} (\tau) \right| \leq  \epsilon(x_0) \to 0,\quad \text{ as } x_0 \to 0,
\end{equation}
where $\hat v_{K_0} (\tau)  = \left( (1 - \tau) + \frac{(p-1) K_0^2}{4 p}\right)^{-\frac{1}{p-1}}$.\\
From \eqref{relation-x-and-xi-tau} and \eqref{sup-v-xi-tau-apro-1}, we deduce
\begin{equation}\label{limit-u-start}
u^* (x_0) = \lim_{t \to T} u(x_0, t) = \psi(t_0(x_0)) \lim_{\tau \to 1} v (x_0, 0, \tau) \sim \psi(t_0(x_0)) \left(\frac{p -  1}{ 4 p} \right)^{- \frac{1}{p-1}}.
\end{equation} 
Using the relation \eqref{equa-x-0-and-t-0-x-0}, we find that
\begin{equation}\label{asymp-T-t-0-x-0} 
  T  - t_0 \sim \frac{|x_0|^2}{ 2 K_0 |\ln |x_0||} \text{ and  }\ln(T - t_0(x_0)) \sim 2 \ln (|x_0|), \quad \text{ as } x_0 \to 0,
  \end{equation}
The formula \eqref{sharp-u-*} then follows from Lemma  \ref{assymptotic-psi-T}, \eqref{limit-u-start} and \eqref{asymp-T-t-0-x-0}. This concludes the proof of Theorem \ref{existence}, assuming that Proposition \ref{prop:redu} holds.
 \end{proof}

\section{Proof of Proposition \ref{prop:redu}.}
This section is devoted to the proof of Proposition \ref{prop:redu},
 which is the heart of our analysis. We proceed into two parts. In the first part, we derive \textit{a priori estimates} on $q(s)$ in $S_A(s)$. In the second part, we show that the new bounds are better than those defined in $S_A(s)$, except for the first two components $(q_0, q_1)$. This means that the problem is reduced to the control of a finite dimensional function $(q_0,q_1)$, which is the conclusion of item $(i)$ of Proposition \ref{prop:redu}. Item $(ii)$ of Proposition \ref{prop:redu} is just direct consequence of the dynamics of the modes $q_0$ and $q_1$. Let us start the first part.

\subsection{A priori estimates on $q(s)$ in $S_A(s)$.}
In this part we derive the \textit{a priori estimates} on the components
 $q_{2}, q_-, q_e$ which implies the conclusion of Proposition \ref{prop:redu}.
 Firstly, let us give some dynamics of $q_0, q_1 = (q_{1,i})_{1 \leq i \leq n}$ and $q_2 = (q_{2,i,j})_{1 \leq i,j \leq n}$. More precisely, we claim the following. 

\begin{proposition}[Dynamics of equation \eqref{equa-q}]\label{prop:dyn} There exists $A_4\geq 1,$ such that $\forall A \geq A_4$ there exists $s_4(A)\geq 1$, such that the following holds for all $s_0 \geq s_4(A)$: Assume that for all $s \in [s_0,s_1]$ for some $s_1\geq  s_0$, $q(s) \in S_A(s)$, then the following holds for all $s \in [s_0,s_1]$:\\
\noindent $(i)$ (ODE satisfied by the positive and null modes)
\begin{equation}\label{odeq0q1}
m = 0, 1, \quad \left|q_m' (s) - \left(1 - \frac{m}{2}\right)q_m(s)\right| \leq \frac{C}{s^2},
\end{equation}
and 
\begin{equation}\label{odeq2}
\left|q_2'(s)+ \frac{2}{s} q_2(s)\right| \leq \frac{C\ln s}{s^3}.
\end{equation}
\noindent $(ii)$ (Control of the negative and outer parts) 
\begin{align}
\left\|\frac{q_-(y,s)}{1 + |y|^3}\right\|_{L^\infty} &\leq \frac{C}{s^2}\left((s - \sigma) + e^{-\frac{s - \sigma}{2}}A + e^{-(s-\sigma)^2}A^2\right), \label{conq-} \\
\left\|q_e(s)\right\|_{L^\infty} &\leq \frac{C}{\sqrt{s}} \left((s - \sigma) + A^2e^{-\frac{s - \sigma}{p}} + Ae^{s - \sigma}\right).\label{conqe}
\end{align}
\end{proposition}
\begin{proof} We proceed in two parts: \\
- In the first part we project equation \eqref{equa-q} to write ODEs satisfied by $q_m$ for $m = 0, 1,2$.\\
- In the second part we use the integral form of equation \eqref{equa-q} and the dynamics of the linear operator $\mathcal{L} + V$ to derive a priori estimates on $q_-$ and $q_e$. \\

\noindent -\textit{Part 1: ODEs satisfying by the positive and null modes.} We give the proof of \eqref{odeq0q1} and \eqref{odeq2} in this part. We only deal with the proof of \eqref{odeq2} because the same proof holds for \eqref{odeq0q1}. By formula \eqref{defini-of_q_2} and equation \eqref{equa-q}, we write for each $1 \leq i,j \leq n$, 
\begin{equation}\label{remainder-q-2}
\left|q_{2,i,j}'(s)  -  \int \left[ \mathcal{L} q + Vq   + B(q) + R(y,s)   + D(q,s) \right] \chi  \left( \frac{y_i y_j }{4} - \frac{\delta_{i,j}}{2} \right) \rho dy  \right|  \leq C e^{-s}.
\end{equation}
Using the assumption $q(s) \in S_A(s)$ for all  $s \in [s_0, s_1]$, we derive the following estimates for all $s \in [s_0,s_1]$: 
$$ \left|\int   \mathcal{L}(q)\chi  \left( \frac{y_i y_j }{4} - \frac{\delta_{i,j}}{2} \right)\rho dy  \right| \leq \frac{C}{s^3},$$
from Lemmas   \ref{inside-D}, \ref{lemma-of-rest-termes} and \ref{estimate-B-q}
\begin{eqnarray*}
 \left| \int   V q \chi  \left( \frac{y_i y_j }{4} - \frac{\delta_{i,j}}{2} \right) \rho  dy  + \frac{2 }{s} q_{2,i,j}(s)\right| &\leq & \frac{C A}{s^3},\\
 \left|\int B(q) \chi  \left( \frac{y_i y_j }{4} - \frac{\delta_{i,j}}{2} \right) \rho dy \right| &\leq & \frac{C }{s^3},\\
\left|  \int  R \chi  \left( \frac{y_i y_j }{4} - \frac{\delta_{i,j}}{2} \right) \rho  dy \right| &\leq & \frac{C}{s^3}, \\
\left|  \int D(q, s) \chi  \left( \frac{y_i y_j }{4} - \frac{\delta_{i,j}}{2} \right) \rho dy  \right| &\leq & \frac{C\ln s}{s^3}.
\end{eqnarray*}
Gathering all these above estimates to \eqref{remainder-q-2} yields 
$$\left|q'_{2,i,j} + \frac{2}{s}q_{2,i,j}\right| \leq \frac{C\ln s}{s^3},$$
which concludes the proof of \eqref{odeq2}.\\

\noindent -\textit{Part 2: Control of the negative and outer parts.} We give the proof of \eqref{conq-} and \eqref{conqe} in this part. the control of $q_-$ and $q_e$ is mainly based on the dynamics of the linear operator $\mathcal{L} + V$. In particular,  we use the following integral form of equation \eqref{equa-q}: for each $s \geq \sigma \geq s_0$,
\begin{equation}\label{Duhamel-q}
q(s)  = \mathcal{K}(s,\sigma) q(\sigma) + \int_{\sigma}^s \mathcal{K}(s,\tau) \left[ B(q)(\tau)  + R(\tau) + D(q,\tau)\right] d \tau = \sum_{i=1}^4 \vartheta_i(s,\sigma),
\end{equation}
where $\{\mathcal{K}(s,\sigma)\}_{s \geq \sigma}$ is defined by
\begin{equation}\label{fundamental-sol}
\left\{ \begin{array}{l}
\partial_s \mathcal{K}(s,\sigma)  = (\mathcal{L} + V)  \mathcal{K}(s,\sigma),\quad  s > \sigma,\\
\mathcal{K}(\sigma,\sigma) = Id,
\end{array}
\right.
\end{equation}
and 
\begin{align*}
\vartheta_1(s,\sigma) &= \mathcal{K}(s,\sigma) q(\sigma), \quad \vartheta_2(s,\sigma) = \int_{\sigma}^s \mathcal{K}(s,\tau) B(q)(\tau) d \tau,\\
\vartheta_3(s,\sigma) &= \int_{\sigma}^s \mathcal{K}(s,\tau) R(.,\tau) d \tau, \quad \vartheta_4(s,\sigma) = \int_{\sigma}^s \mathcal{K}(s,\tau) D(q,\tau) d \tau.
\end{align*}

From \eqref{Duhamel-q}, it is clear to see the strong influence of the kernel $\mathcal{K}$ in this formula. It is therefore convenient to recall the following result which the dynamics of the linear operator $\mathcal{K} = \mathcal{L} + V$. 

\begin{lemma}\label{dynamic-K-feym}\textbf(A priori estimates of the linearized operator in the  decomposition in \eqref{decomp-q2}). For all $\rho^* \geq 0$, there exists $s_5(\rho^*)  \geq 1$, such that if $\sigma \geq s_5(\rho^*)$ and $v \in L^{2}_{\rho}$ satisfying
\begin{equation}\label{condition-q-sigma}
\sum_{m=0}^2 |v_m| + \left\|\frac{v_-}{1 + |y|^3}\right\|_{L^{\infty}}   +\|v_e\|_{L^{\infty}} < \infty.
\end{equation} 
Then, $\forall s \in [\sigma, \sigma + \rho^*]$  the function $\theta(s) = \mathcal{K}(s,\sigma) v $ satisfies 
\begin{equation}\label{control-K-q-sigma1}
\begin{array}{l}
\left\|\frac{\theta_-(y,s)}{1 + |y|^3}\right\|_{L^{\infty}} \leq \frac{C e^{s -\sigma} \left( (s - \sigma)^2   +1 \right)}{s} \left(  |v_0|  + |v_1|  + \sqrt s |v_2|\right)\\
\quad \quad \quad \quad \quad + C e^{-\frac{(s-\sigma)}{2}} \left\|\frac{v_-}{1 + |y|^3}\right\|_{L^{\infty}} + C \frac{e^{-(s-\sigma)^2}  }{s^{\frac{3}{2}}} \|v_e\|_{L^{\infty}},
\end{array}
\end{equation}
and
\begin{equation}\label{control-K-q-e}
\|\theta_e(y,s)\|_{L^{\infty}} \leq  C e^{s -\sigma} \left(  \sum_{l=0}^2 s^{\frac{l}{2}} |v_l|    +s^{\frac{3}{2}} \left\|\frac{v_-}{1  +|y|^3}\right\|_{L^{\infty}}\right) + C e^{-\frac{s -\sigma}{p}} \|v_e\|_{L^{\infty}}.
\end{equation}
\end{lemma}
\begin{proof} The proof of this result was given by Bricmont and Kupiainen \cite{BKnon94} in one dimensional case. It was then extended in higher dimensional case in \cite{NZens16}. We kindly refer interested readers to Lemma 2.9 in \cite{NZens16} for a detail of the proof.
\end{proof}

In view of formula \eqref{Duhamel-q}, we see that Lemma \eqref{dynamic-K-feym} plays an important role in deriving the new bounds on the components $q_-$ and $q_e$. Indeed, given bounds on the components of $q$, $B(q)$, $D(q)$ and $R$, we directly apply Lemma \ref{dynamic-K-feym} with $\mathcal{K}(s, \sigma)$ replaced by $\mathcal{K}(s, \tau)$ and then integrating over $\tau$ to obtain estimates on $q_-$ and $q_e$. In particular, we claim the following which immediately follows \eqref{conq-} and \eqref{conqe} by addition. 

\begin{lemma}\label{control-prin-q-e-q-}
For all $\tilde A \geq 1, A \geq 1, \rho^* \geq 0$,  there exists $s_6(A, \rho^*) \geq 1$ such that $\forall s_0 \geq s_6(A,\rho^*)$ and $q(s) \in S_A(s), \forall s \in [\sigma, \sigma + \rho^*] \text{ where } \sigma \geq s_0$. Then, we have the following properties:
\begin{itemize}
\item[$a)$ ] Case $\sigma  \geq s_0$: for all $ s \in [\sigma, \sigma + \rho^*] $,
\begin{itemize}
\item[$i)$] (The linear term $\vartheta_1(s,\sigma)$)
\begin{eqnarray*}
\left\|\frac{(\vartheta_1(s,\sigma))_-}{1 + |y|^3} \right\|_{L^{\infty}} &\leq & C  \frac{\left( 1 +  e^{-\frac{s  -\sigma}{2}} A + e^{-(s -\sigma)^2} A^2\right)}{s^2},\\
\| (\vartheta_1(s,\sigma))_e\|_{L^{\infty}} &\leq & C \frac{A^2 e^{- \frac{s -\sigma}{p} } + A e^{s -\sigma}}{s^{\frac{1}{2}}}.
\end{eqnarray*}
\item[$ii)$] (The quadratic term $\vartheta_2(s,\sigma)$)
\begin{eqnarray*}
\left\|\frac{(\vartheta_2(s,\sigma))_-}{1 + |y|^3} \right\|_{L^{\infty}} &\leq & \frac{C(s - \sigma)}{ s^{2 + \epsilon}}, \quad \| (\vartheta_2(s,\sigma))_e\|_{L^{\infty}} \leq  \frac{C (s - \sigma)}{ s^{\frac{1}{2}  +\epsilon}}.
\end{eqnarray*}
where $\epsilon = \epsilon(p) > 0$.
 \item[$iii)$] (The correction term $\vartheta_3(s,\sigma)$ )
\begin{eqnarray*}
\left\|\frac{(\vartheta_3(s,\sigma))_-}{1 + |y|^3} \right\|_{L^{\infty}} &\leq & \frac{C (s  - \sigma) }{ s^{2}}, \quad \| (\vartheta_3(s,\sigma))_e\|_{L^{\infty}} \leq  \frac{C (s  -\sigma ) }{ s^{\frac{3}{4} }}.
\end{eqnarray*}
\item[$iv)$] (The nonlinear term $\vartheta_4(s,\sigma)$)
\begin{equation*}
\left\|\frac{(\vartheta_4(s,\sigma))_-}{1 + |y|^3} \right\|_{L^{\infty}} \leq  \frac{C (s - \sigma)}{ s^{2 }}, \quad \| (\vartheta_4(s,\sigma))_e\|_{L^{\infty}} \leq  \frac{C (s -\sigma ) }{ s^{\frac{3}{4}  }}.
\end{equation*}
 \end{itemize}
\item[$b)$] Case $\sigma  =  s_0$, we assume in addition
\begin{align*}
|q_m(s_0)|  \leq \frac{\tilde A}{s_0^2}, \quad |q_2(s_0)| \leq \frac{\tilde A \ln^2 s_0}{s_0^2},\\
\left\|\frac{q_-(y,s_0)}{1 + |y|^3}\right\|_{L^{\infty}} \leq \frac{\tilde A}{s_0^2}, \quad \|q_e(s_0)\|_{L^{\infty}} \leq \frac{\tilde A}{\sqrt{s_0}}.
\end{align*}
Then, for  all $s \in [s_0,s_0 + \rho^*]$ we have $a)$ and the following properties:
\begin{eqnarray*}
\left\|\frac{(\vartheta_1(s,s_0))_-}{1  + |y|^3}\right\|_{L^{\infty}} &\leq & \frac{C \tilde A}{s^2}, \quad \|(\vartheta_1(s,s_0))_e\|_{L^{\infty}} \leq  \frac{C \tilde A (1 + e^{s - s_0})}{\sqrt s}.
\end{eqnarray*}
\end{itemize}
\end{lemma}
\begin{proof} The proof simply follows from definition of the set $S_A$ and Lemma \ref{dynamic-K-feym}. 
In particular, we make use Lemmas \ref{inside-D} , \ref{lemma-of-rest-termes} and \ref{estimate-B-q}  to derive the bounds on the components of the term $B$, $D$ and $R$ as follows:
$$
\sum_{m \in \mathbb{N}^n, |m| = 0}^2\left| B(q)_{m}(s)\right| \leq  \frac{C}{s^3}, \quad \left\| \frac{B(q)_{-}(s)}{1 + |y|^3}\right\|_{L^{\infty}}  \leq   \frac{C  }{s^{2 + \epsilon }}, \quad \left\| B(q)_{e}(s) \right\|_{L^{\infty}} \leq \frac{C }{s^{\frac{1}{2} + \epsilon}},
$$
and
$$
\sum_{m \in \mathbb{N}^n, |m| = 0}^2\left| R_{m}(s)\right| \leq  \frac{C}{s^2}, \quad \left\| \frac{R_{-}(s)}{1 + |y|^3}\right\|_{L^{\infty}}  \leq   \frac{C  }{s^{2 + \frac{1}{2} }}, \quad \left\| R_{e}(s) \right\|_{L^{\infty}} \leq \frac{C}{s^{\frac{3}{4} }},
$$
and
$$
\sum_{m \in \mathbb{N}^n, |m| = 0}^2\left| D(q)_{m}(s)\right| + \left\|\frac{D(q)_{-}(s)}{1 + |y|^3}\right\|_{L^{\infty}} \leq  \frac{C\ln s}{s^3}, \quad \left\| D(q)_{e}(s) \right\|_{L^{\infty}} \leq \frac{C}{s^{\frac{3}{4}}},
$$
where $\epsilon = \epsilon (p) > 0$.  We simply inject these bounds to the a priori estimates given in  Lemma \ref{dynamic-K-feym} to obtain the bounds on $\big(\vartheta_m)_-$ and $\big(\vartheta_m\big)_e$ for $m = 2, 3, 4$. The estimate on $\vartheta_1$ directly follows from Lemma \ref{dynamic-K-feym} and the assumption $q(s) \in S_A(s)$. This ends the proof of Lemma \ref{control-prin-q-e-q-}.
\end{proof}
From the formula \eqref{Duhamel-q}, the estimates \eqref{conq-} and \eqref{conqe} simply follows from Lemma \ref{control-prin-q-e-q-} by addition. This concludes the proof of Proposition \ref{prop:dyn}.
\end{proof}

\subsection{Conclusion of Proposition \ref{prop:redu}.}
In this part, we  give the proof of Proposition \ref{prop:redu} which 
is a consequence of the dynamics of equation \eqref{equa-q} given in 
Proposition \ref{prop:dyn}. Indeed,  
the item $(i)$ of Proposition \ref{prop:redu} directly follows from the following result.
 \begin{proposition}[Control of $q(s)$ by $(q_0, q_1)(s)$ in $S_A(s)$]\label{control-q(s)}
There exists $A_7\geq 1$ such that $\forall A \geq A_7$, there exists $s_7(A)\geq 1$ such that for all $s_0 \geq s_7(A)$, we have the following properties: 
\begin{itemize}
\item[$a)$] $q(s_0) = \psi_{d_0,d_1,s_0}(y)$, where  $(d_0,d_1) \in \mathbb{D}_{A,s_0}$,
\item[$b)$] For all $s \in [s_0,s_1]$, $q(s) \in S_A(s)$.
\end{itemize}
Then for all $s \in [s_0,s_1]$, we have
\begin{eqnarray}
 \forall i,j \in \{1, \cdots, n\}, \quad |q_{2,i,j}(s)| &< & \frac{A^2 \ln^2 s}{ s^2}, \label{conq2} \\
\left\| \frac{q_-(y,s)}{1 + |y|^3}\right\|_{L^{\infty}} \leq   \frac{A}{2 s^{2}}, &\quad& \|q_e(s)\|_{L^{\infty}} \leq  \frac{A^2}{2 \sqrt s}, \label{conq-qe}
\end{eqnarray}
where $\mathbb{D}_{A,s_0}$ is introduced in Proposition \ref{pro-initial} and $\psi_{d_0, d_1}$ is defined as in \eqref{def-psi0}.
\end{proposition}
\begin{proof} Since the proof of \eqref{conq-qe} is similar to the one written in \cite{MZdm97}, we only deal with the proof of \eqref{conq2} and refer to Proposition 3.7 in \cite{MZdm97} for the proof of \eqref{conq-qe}. We argue by contradiction to prove \eqref{conq2}. Let $i, j \in \{1, \cdots, n\}$ and assume that there is $s_* \in [s_0, s_1]$ such that 
$$\forall s \in [s_0, s_*), \quad |q_{2,i,j}(s)| < \frac{A^2 \ln^2(s)}{s^2} \quad \text{and} \quad |q_{2,i,j}(s_*)| = \frac{A^2 \ln^2(s_*)}{s_*^2}.$$
Assuming that $q_{2,i,j}(s_*) > 0$ (the negative case is similar), we have on the one hand
$$q'_{2,i,j}(s_*) \geq \frac{d}{ds}\left(\frac{A^2 \ln^2 s}{s^2}\right)_{s = s_*} = \frac{2A^2 \ln s_*}{s_*^3} - \frac{2A^2 \ln^2 s_*}{s_*^3}.$$
On the other hand, we have from \eqref{odeq2}, 
$$q'_{2,i,j}(s_*) \leq - \frac{2A^2 \ln^2 s_*}{s_*^3} + \frac{C \ln s_*}{s_*^3}.$$
The contradiction then follows if $2A^2 > C$. This concludes the proof of Proposition \ref{control-q(s)}.
\end{proof}

From Proposition \ref{control-q(s)}, we see that if $q(s) \in \partial S_A(s_1)$, 
the first two components $(q_0,q_1)(s_1)$ must be in $\partial \hat S_A(s_1)$, which is the conclusion of item $(i)$ of Proposition \ref{prop:redu}.

The proof of item $(ii)$ of Proposition \ref{prop:redu} follows from \eqref{odeq0q1}. Indeed, it is easy to see from \eqref{odeq0q1} that for all  $ i \in \{ 1,...,n\}$ and for each $ \varepsilon_0, \varepsilon_i  = \pm 1$, then if $q_0 (s_1) = \varepsilon_0 \dfrac{A}{s_1^2}$ and $q_{1,i} (s_1) = \varepsilon_i \dfrac{A}{s^2_1}$, it follows that the sign of $\dfrac{{dq_0 }}{{ds}}\left( {s_1 } \right)$ and  $\dfrac{{dq_{1,i} }}{{ds}}\left( {s_1 } \right)$  are opposite the sign of $\dfrac{d}{{ds}}\left({\dfrac{{\varepsilon_0 A}}{{s^2 }}}\right)\left( {s_1 } \right)$ and $\dfrac{d}{{ds}}\left( {\dfrac{{\varepsilon_i A}}{{s^2 }}} \right)\left( {s_1 } \right)$ respectively. Hence, $(q_0,q_1)(s)$ will actually leave $\hat S_{A} (s)$ at $s_1 \geq s_0$ for $s_0$ large enough. This concludes the proof of Proposition \ref{prop:redu}.

 \appendix
 \section{Some elementary lemmas.}
 
 \begin{lemma}\label{assymptotic-psi-T}
 For each $T > 0$, there exists only one positive solution of equation \eqref{euqq-ODE-psi}. Moreover, the solution $\psi$ satisfies the following asymptotic:
 \begin{equation}\label{inequ-psi-ep-+}
 \psi
(t) \sim \kappa_{\alpha}(T - t)^{-\frac{1}{p-1}} |\ln(T - t)|^{-\frac{\alpha}{p-1}}, \text{ as } t \to T,
 \end{equation}
 where $\kappa_{\alpha} = (p-1)^{-\frac{1}{p-1}} \left( \frac{p-1}{2} \right)^{\frac{\alpha}{p-1}}$.
 \end{lemma}
 \begin{proof} Consider the ODE
$$\psi' = \psi^p \ln^{\alpha}(\psi^2  +2), \quad \psi(0) > 0.$$
The uniqueness and local existence are derived by the Cauchy-Lipschitz property.  
Let $T_{max}, T_{min}$ be the maximum and minimum  time  of the existence of the positive solution, i.e.  $\psi(t)$ exists for all $t \in (T_{min}, T_{max})$. We now prove that $T_{max} < + \infty$ and $T_{min} = - \infty$. By contradiction, we suppose that  the solution exists on $[0, + \infty)$, we have 
 $$\lim_{t_1 \to +\infty} \int_0^{t_1} \frac{\psi'}{\psi^p\ln^{\alpha}(\psi^2  +2)} dt = \lim_{t_1 \to +\infty} \int_0^{t_1}dt = +\infty.$$  
Since $\int_0^{t_1} \frac{\psi'}{\psi^p\ln^{\alpha}(\psi^2  + 2)} dt$ is bounded, the contradiction then follows.  With  a similar  argument we can prove that $T_{min} = - \infty$.
 Let us now prove \eqref{inequ-psi-ep-+}. We deduce from \eqref{euqq-ODE-psi} that
$$T - t =  \int_{\psi(t)}^{+\infty}\frac{du}{ u^p \ln^{\alpha}(u^2  +2)}.$$
Thus, for all $\delta \in (0, p-1)$, there exist $t_\delta$ such that for all $t \in (t_\delta, T)$, we have
$$ \int_{\psi(t)}^{+\infty} \frac{du}{u^{p + \delta}} \leq T -t \leq  \int_{\psi(t)}^{+\infty} \frac{du}{u^{p  -  \delta}}.$$
This follows for all $t \in (t_\delta, T)$:
$$ (p-1 + \delta)^{-\frac{1}{p - 1 + \delta}} (T - t)^{-\frac{1}{ p - 1 + \delta}} \leq \psi(t) \leq (p-1 - \delta)^{-\frac{1}{p - 1 - \delta}} (T - t)^{-\frac{1}{ p - 1 - \delta}},$$
from which we have
$$\ln \psi (t)  \sim - \frac{1}{p-1} \ln(T  -t) \quad  \text{ as } \;\; t \to T,$$
and
$$\ln (\psi^2 + 2) \sim -\frac{2}{p-1} \ln(T-t) \quad \text{ as }\;\; t \to T.$$
Hence, we obtain
\begin{equation}\label{equi-ODE-psi}
\psi' = \psi^p \ln(\psi^2 +2) \sim \psi^p \left[ - \frac{2}{p-1} \ln(T - t)\right]^{\alpha} \quad \text{ as }\;\; t \to T,
\end{equation}
which yields
$$ \frac{\psi'}{\psi^p}  \sim   \left( \frac{2}{p-1}\right)^{\alpha} |\ln(T - t)|^{\alpha}\quad \text{ as } \;\; t \to T.$$
This implies 
$$\frac{1}{p-1} \psi^{1-p}  \sim \left( \frac{2}{p-1}\right)^{\alpha} \int^T_t |\ln(T - v)|^{\alpha} dv \sim \left( \frac{2}{p-1}\right)^{\alpha} (T - t) |\ln(T - t)|^{\alpha} \quad \text{ as }\; t \to T, $$
which concludes the proof of \eqref{inequ-psi-ep-+}.
\end{proof}

\begin{lemma}\label{lemma-inequality-integral_I-h}
For all $\alpha \in (0,1), \theta > 0$ and $0 < h < 1$, the integral
$$I (h)  = \int_h^1  (s - h)^{- \alpha} s^{-\theta} ds $$
satisfies:
\begin{itemize}
\item[$i)$] if $\alpha  + \theta > 1$, then 
$$ I(h) \leq \left( \frac{1}{1 - \alpha} + \frac{1}{ \alpha  +\theta  - 1} \right) h^{1 - \alpha  - \theta}.$$
\item[$ii)$] If $\alpha  + \theta   = 1$, then
$$I(h)  \leq  \frac{1}{ 1 - \alpha}  + |\ln h|.$$
\item[$iii)$] If $\alpha + \theta < 1$, then
$$I(h) \leq   \frac{1}{ 1 - \alpha - \theta}.$$
\end{itemize}
\end{lemma}
\begin{proof}
See Lemma 2.2 of Giga and Kohn \cite{GKcpam89}
\end{proof} 
\begin{lemma}[A version of Gronwall Lemma]\label{inequa-3-function-y-r-h}
If $y(t), r(t)$ and $q(t)$  are continuous functions defined  on $[t_0, t_1]$ such that 
$$ y(t) \leq  y_0  + \int_{t_0}^{t} y(s) r(s) ds  +  + \int_{t_0}^t h(s) ds, \forall t \in [t_0, t_1].$$
Then,
$$y(t) \leq \displaystyle e^{\displaystyle\int_{t_0}^t r(s)ds } \left[ y_0  + \int_{t_0}^t h(s) e^{ -\displaystyle \int_{t_0}^s r(\tau) d\tau } ds\right].$$
\end{lemma}
 \begin{proof}
 See Lemma 2.3 of Giga and Kohn \cite{GKcpam89}.
 \end{proof}
\begin{lemma}\label{citia-not-blow-up} For each $T_2 < T, \delta > 0 $. There exists $\epsilon  = \epsilon (T,T_2, \delta, n, p) > 0$ such that for each $v(x,t)$ satisfying
 \begin{equation}\label{inequa-ODE-v-gene}
 \left| \partial_tv - \Delta v\right| \leq  C |v|^p \ln^{\alpha}(v^2 + 2), \quad \forall |x| \leq \delta, \;\; t \in (T_2, T), \delta > 0,
 \end{equation}
 and
\begin{equation}\label{the-condi-not-blow}
| v(x,t)| \leq \epsilon \psi(t), \quad \forall |x| \leq \delta, \quad t \in (T_2, T),
\end{equation}
where  $\psi (t)$ is  the unique positive solution of \eqref{euqq-ODE-psi}. Then, $v(x,t)$ does not blow up at $(0,T)$.
 \end{lemma}
\begin{proof} Since the argument is almost the same as in \cite{GKcpam89} treated for the case $\alpha = 0$, we only sketch the main step for the sake of completeness. Let $\phi \in C^{\infty} (\mathbb{R}^n), \phi = 1 \text{ if } |x| \leq \frac{\delta}{2}, \phi = 0 \text{ if } |x| \geq \delta$, and consider $\omega  = \phi v$ satisfying
\begin{equation}\label{parome-lap=fv-g}
\partial_t \omega - \Delta \omega = f \phi + g,
\end{equation}
where
$$f = \partial_t v - \Delta v \quad  \text{ and } \quad g = v \Delta \phi -  2 \nabla . (v \nabla \phi).$$
By using the Duhamel's formula, we write
\begin{equation}\label{Duhamel-ome-lem-appen}
\omega(t) = e^{(t - T_2)\Delta} (\omega(T_2)) + \int_{T_2}^{t}  \left(e^{(t - \tau)\Delta} (\phi f) +  e^{(t - \tau)\Delta}(g) \right) d\tau, \forall t \in [T_2, T),
\end{equation}
where $e^{t\Delta}$ is the heat semigroup satisfying the following properties: for all $h \in L^\infty$,
$$\| e^{t \Delta} h \|_{L^{\infty}} \leq  \|  h \|_{L^{\infty}} \text{ and } \| e^{t \Delta} \nabla h \|_{L^{\infty}} \leq  C t^{- \frac{1}{2}}\|  h \|_{L^{\infty}}, \forall t >0.$$
The formula \eqref{Duhamel-ome-lem-appen} then yields 
\begin{align}
\| \omega(t) \|_{L^{\infty}} &\leq  C + C\int_{T_2}^t \|\omega(\tau)\|_{L^{\infty}} \||v|^{p-1} \ln^{\alpha}( v^2  +2)(\tau)\|_{L^{\infty}(|x| \leq \delta)}\nonumber\\
& \quad + C\int_{T_2}^t (t - \tau)^{-\frac{1}{2}}\|v(\tau)\|_{L^{\infty}(|x| \leq \delta)} d\tau,\label{inequa-gen-of-ome}
\end{align}
for some constant $C = C(n, p, \phi, T, T_2,  \delta) > 0$.

From \eqref{inequa-ODE-v-gene}, \eqref{the-condi-not-blow} and Lemma \eqref{assymptotic-psi-T}, we find that  for all $|x| \leq \delta$, and $\tau \in [T_2, T)$, 
$$|v(\tau)|^{p-1} \ln^{\alpha}( v^2(\tau)  +2) \leq C \psi^{p-1}(\tau) \ln^{\alpha}(\psi^2(\tau) + 2 ) \leq C (T - \tau)^{-1},$$
and
$$|v(\tau) | \leq C (T -\tau)^{-\frac{1}{p-1}} |\ln(T -\tau)|^{- \frac{\alpha}{p-1}}.$$
The estimate \eqref{inequa-gen-of-ome} becomes
\begin{align}
\| \omega(t) \|_{L^{\infty}} &\leq  C + C\epsilon^{p-1} \int_{T_2}^t  ( T-\tau)^{-1}\|\omega(\tau)\|_{L^{\infty}}  d\tau \nonumber \\
&\quad + C\epsilon\int_{T_2}^t (t - \tau)^{-\frac{1}{2}}(T -\tau)^{-\frac{1}{p-1}} |\ln(T -\tau)|^{-\frac{\alpha}{p-1}} d\tau. \label{inequa-gen-of-ome-rewri-1}
\end{align}
In particular, we now consider   $0 < \lambda  \ll  \frac{1}{2}$ fixed, then  we have:
$$ (T - \tau )^{- \frac{1}{p-1}} |\ln(T-\tau)|^{-\frac{\alpha}{p-1}}  \leq C(\alpha, \lambda) (T - \tau)^{- \left( \frac{1}{p-1} + \lambda \right)}, \forall \tau \in (T_2, T).$$
Hence, we rewrite \eqref{inequa-gen-of-ome-rewri-1} as follows
\begin{align}
\| \omega(t) \|_{L^{\infty}} &\leq  C + C \epsilon^{p-1} \int_{T_2}^t  ( T-\tau)^{-1}\|\omega(\tau)\|_{L^{\infty}}  d\tau \nonumber \\
&\quad + C\epsilon\int_{T_2}^t (t - \tau)^{-\frac{1}{2}}(T -\tau)^{-\left (\frac{1}{p-1}  + \lambda\right)}  d\tau, \label{inequa-gen-of-ome-rewri-2}
\end{align}
where  $C(n,p, \phi, \alpha, \epsilon, \lambda,p)$. Beside that, by changing variables  $s =  T - \tau, h = T-  t$ we have
\begin{equation}\label{changing-variables-to-I-h}
\int_{T_2}^t (t - \tau)^{-\frac{1}{2}}(T -\tau)^{- \theta(p, \lambda)}  d\tau = \int^{T - T_2}_h (s - h)^{-\frac{1}{2}} (s)^{  - \theta(p,\lambda)} ds,
\end{equation}
where $ \theta(p,\lambda) = \left (\frac{1}{p-1}  + \lambda\right)$.

\textbf{Case 1:} If $ \theta(p,\lambda) < \frac{1}{2}$, by using  $iii)$ of Lemma \ref{lemma-inequality-integral_I-h} we deduce from \eqref{inequa-gen-of-ome-rewri-2}, \eqref{changing-variables-to-I-h} that 
$$ \|\omega (t)\|_{L^{\infty}} \leq  C   + C \epsilon^{p-1} \int_{T_2} ^t  (T  - s)^{-1}  \|\omega (s)\|_{L^{\infty}}ds, $$
Therefore, by Lemma  \ref{inequa-3-function-y-r-h}, 
\begin{equation}\label{estimates-norm-ome-K-ep-p-1}
\|\omega (t)\|_{L^{\infty}}  \leq  C (T - t)^{ - C\epsilon^{p-1}},
\end{equation}
Choosing $\epsilon$ small enough such that $C \epsilon^{p-1} \leq \frac{1}{2 (p-1)}$. Then, we conclude from \eqref{estimates-norm-ome-K-ep-p-1} that 
\begin{equation}\label{estimates-q--leq-12-p-1-T-t}
|v (x, t)| \leq C (T - t) ^{- \frac{1}{2 (p-1)}}, \text{ for } |x| \leq \frac{1}{2}, t \leq T.
\end{equation}
By using parabolic regularity  theory and the same argument as in Lemma 3.3 of \cite{GKiumj87}, we can prove that \eqref{estimates-q--leq-12-p-1-T-t} actually prevents blowup.

\textbf{Case 2:} $\theta (\lambda, p) = \frac{1}{2}, $ it is similar to the first case, by using $ii)$ of Lemma \ref{lemma-inequality-integral_I-h},  \eqref{inequa-gen-of-ome-rewri-2} and  \eqref{changing-variables-to-I-h} we yield
$$ \|\omega (t)\|_{L^{\infty}} \leq  C (1 + |\ln (T -t)|)   + C \epsilon^{p-1} \int_{T_2} ^t  (T  - s)^{-1}  \|\omega (s)\|_{L^{\infty}}ds, $$
However, we derive from Lemma  \ref{inequa-3-function-y-r-h} that
\begin{equation}\label{estima-omega-the-lamb=1-2}
\|\omega (t)\|_{L^{\infty}} \leq C (T -t)^{-K\epsilon^{p-1}}, 
\end{equation}
where $C =C(n,p,\phi, T,T_2, \delta) $. We now take $\epsilon$ is small enough such that $C \epsilon^{p-1} \leq \frac{1}{2 (p-1)}$, which follows \eqref{estimates-q--leq-12-p-1-T-t}.

\textbf{Case 3:} $\theta (\lambda, p) > \frac{1}{2}, $ by using Lemmas \ref{lemma-inequality-integral_I-h},  \ref{inequa-3-function-y-r-h} and arguments similar to obtain 
$$|v(x,t)| \leq C (T -t)^{\frac{1}{2} - \theta(p,\lambda)}, \quad \forall |x| \leq \delta,\; t \in [T_2, T),$$
Repeating the step in finite steps would end up with \eqref{estimates-q--leq-12-p-1-T-t}. This concludes the proof of Lemma \ref{citia-not-blow-up}.
\end{proof}

The following lemma gives the asymptotic behavior of $h(s)$ ans $\psi_1(s) $ defined in \eqref{def-psi(s)} and \eqref{defini-psi-1-in-s}.

\begin{lemma}\label{function-h(s)}
Let  $h(s)$ and $\psi_1(s) $ be defined as in \eqref{def-psi(s)} and \eqref{defini-psi-1-in-s}  respectively.  Then we have 
\begin{itemize}
\item[$i)$] 
\begin{equation}\label{lnpsi(s)}
\frac{1}{ \ln ( \psi^2_1(s) + 2) }   = \frac{p-1}{2s}  + \frac{\alpha (p-1) \ln s}{ 2 s^2}  + O\left( \frac{1}{s^2} \right), \quad \text{ as } s \to +\infty. 
\end{equation}
\item[$ii)$] 
\begin{equation}\label{hs}
h(s) =  \frac{1}{ p-1} \left[1 - \frac{\alpha}{ s}  - \frac{\alpha^2 \ln s}{s^2}  \right]+ O\left(\frac{1}{s^2} \right), \quad \text{ as } s \to +\infty.
\end{equation}
\end{itemize}
\end{lemma}
\begin{proof} $i)$ Consider $\psi(t)$ the unique positive solution of \eqref{euqq-ODE-psi}. We have 
\begin{equation}\label{T-t=}
T -t  = \int_{\psi(t)}^{+\infty} \frac{dx}{x^p \ln^{\alpha}(x^2  +2)}.
\end{equation}
An integration by parts yields
\begin{equation}\label{equa:2.10}
T -t  =   \displaystyle  \frac{1}{\psi^{p-1}(t) \ln^{\alpha}(  \psi^2(t)  +2)} \left[  \frac{1}{p-1} -  \frac{2 \alpha}{(p-1)^2\ln( \psi^2(t)  +2)}   + O\left( \frac{1}{(\ln^2( \psi^2(t)  +2))}\right) \right].
\end{equation}
Let us write $ \psi (t)  = \psi_1(s)$ where $s = -\log(T-t)$, then we have

\begin{equation}\label{equa:2.12}
\ln ( \psi_1 (s)) = \frac{s}{p-1} - \frac{\alpha}{(p-1)} \ln \left( \ln ( \psi_1( s)) \right)   +O\left( 1\right), \quad \text{ as } s \to + \infty,
\end{equation}
from which,  we deduce that 
\begin{equation}\label{ln s=s-lns-again}
\ln ( \psi_1(s)) = \frac{s}{p-1}  - \frac{\alpha \ln \left( s   \right) }{p-1}   +O(1), \quad \text{ as } s \to + \infty,
\end{equation}
 which is the conclusion $(i)$.

$ii)$   From \eqref{def-psi(s)} and \eqref{equa:2.10}, we have 
\begin{equation}\label{equa:2.11}
h(s)  = \frac{1}{p-1}  - \frac{2 \alpha}{ (p-1)^2 \ln ( \psi^2_1(s) + 2)}  + O \left( \frac{1}{ \ln^2( \psi^2_1(s)   +2)}\right),
\end{equation}
Using \eqref{lnpsi(s)} we conclude the proof of \eqref{hs} as well as Lemma \eqref{function-h(s)}.
\end{proof} 
 \begin{lemma}\label{asymptotic-N-1-N-2}
 Let $N$ be defined as in \eqref{def:N-2}, we have
 \begin{equation}\label{asymp-N-proof}
 N(\bar w, s) =  \frac{ p \bar w^2}{2} + O\left(\frac{|\bar w|\ln s }{s^2}\right) + O\left(\frac{|\bar w|^2}{s}\right) + O(|\bar w|^3) \quad  \text{ as } \quad (\bar w ,s)  \to (0 ,+\infty).
 \end{equation}
 \end{lemma}
\begin{proof} 
 From the definition \eqref{def:N-2} of $N$, let us write
 $$N(\bar w, s) = N_1(\bar w, s) + N_2(\bar w, s),$$
 where
 \begin{align*}
 N_1(\bar w, s) &= h(s) \left( |\bar w  +1|^{p-1} (\bar w  +1) -(\bar w  +1) \right) - \bar w,\\
N_2(\bar w, s) &=   \displaystyle  h(s) |\bar w   +1|^{p-1}(\bar w  +1) \left( \frac{\ln^{\alpha}(\psi_1^2 (\bar w  +1)^2 +2)}{ \ln^{\alpha}( \psi_1^2 +2)} -1 \right).
 \end{align*}
From \eqref{hs} and a Taylor expansion,  we find that 
\begin{equation*}\label{result-asym-N_1}
 N_1(\bar w, s) = \frac{ p \bar w^2}{2}- \frac{\alpha \bar w}{s} + O\left(\frac{|\bar w| \ln s }{s^2}\right) + O\left( \frac{|\bar w|^2}{s}\right) + O(|\bar w|^3) \quad \text{ as } (\bar w ,s)  \to (0 ,+\infty).
 \end{equation*}
 We now claim the following 
 \begin{equation}\label{result-asym-N_2}
 N_2(\bar w, s) = \frac{\alpha \bar w}{s}  + O\left(\frac{|\bar w| \ln s }{s^2}\right) + O(\frac{|\bar w|^2}{s}) \quad \text{ as } (\bar w ,s)  \to (0 ,+\infty),
 \end{equation}
 then, the proof of \eqref{asymp-N-proof} simply follows by addition. 
 
Let us now give the proof of \eqref{result-asym-N_2} to complete the proof of Lemma \ref{asymptotic-N-1-N-2} .  We set
 $$ f(\bar w)  = \ln^{\alpha} ( \psi_1^2 (\bar w +1)^2  + 2), \quad | \bar w|  \leq \frac{1}{2}.$$
 We apply  Taylor expansion to  $f(\bar w)$ at $\bar w = 0$ to find that 
 \begin{equation*}\label{use-expansion-tl-f-z}
 f(\bar w) = \ln^{\alpha}(\psi_1^2  + 2) + 2 \alpha \ln^{\alpha -1}(\psi_1^2  +2) \frac{\psi_1^2 }{ \psi_1^2  +2}\bar w + \frac{f''(\theta)}{2} (\bar w)^2,
 \end{equation*}
where $\theta$ is between $0$ and $\bar w$, and   
\begin{align*}
f''( \theta) & = \alpha (\alpha -1) \ln^{\alpha -2} ( \psi_1^2 (\theta  +1)^2 +2 ) \left(\frac{2 (\theta +1) \psi_1^2}{\psi_1^2 (\theta +1)^2  +2} \right)^2 \nonumber\\
& \quad  +  \alpha \ln^{\alpha -1}(\psi_1^2(\theta +1)^2  +2)\frac{( 4 \psi_1 - 2\psi_1^4 (\theta +1)^2)}{(\psi_1^2 (\theta +1)^2  +2)^2}\label{diffirent-f-ordre-2}. 
\end{align*}
 Since $|\theta| \leq \frac 12$, one can show that
 $$|f''(\theta)| \leq  C \ln^{\alpha - 1}(\psi_1^2 + 2), \quad \forall |\theta| \leq \frac 12.$$
Thus, we have 
$$ f(\bar w) = \ln^{\alpha}(\psi_1^2  +2 )  + 2 \alpha \ln^{\alpha -1}(\psi_1^2  +2)  \bar w  + O \left(|\bar w|^2 \ln^{\alpha -1}(\psi_1^2  +2)\right) + O\left(\frac{|\bar w| \ln^{\alpha -1}(\psi_1^2  + 2)}{\psi_1^2 }\right),$$
as $s \to +\infty$.  This  yields
\begin{equation*}\label{expand-al-lnbarw-ln}
\frac{\ln^{\alpha}(\psi_1^2(\bar w  +1)^2  + 2)}{ \ln^{\alpha}(\psi_1^2 +2)} = 1 + \frac{ 2 \alpha \bar w  }{\ln(\psi_1^2  +2)}  + O\left(\frac{|\bar w|^2}{ \ln(\psi_1^2   +2)}\right)  +O\left(\frac{|\bar w|}{\ln(\psi_1^2 + 2)\psi_1^2 }\right),
\end{equation*}
as $(\bar w, s)  \to (0,+ \infty)$, from which and  \eqref{lnpsi(s)}  we derive 
\begin{equation}\label{expan-al-ln-barw-in-s}
\frac{\ln^{\alpha}(\psi_1^2(\bar w  +1)^2  + 2)}{ \ln^{\alpha}(\psi_1^2(s)  +2)}  - 1 =    \frac{ \alpha (p-1) \bar w }{s} + O\left(\frac{ \ln s |\bar w|}{s^2}\right)+ O\left(\frac{|\bar w|^2}{s}\right).
\end{equation}
From the definition of $N_2$, \eqref{hs}, \eqref{expan-al-ln-barw-in-s} and the fact that  
\begin{equation*}\label{tay-lor-expan-barw-p}
|\bar w  +1|^{p-1} (\bar w  +1) = 1 +  p \bar w + O(|\bar w|^2)  \quad \text{ as } \bar w \to 0,
\end{equation*}
we conclude the proof of \eqref{result-asym-N_2}  as well as Lemma \ref{asymptotic-N-1-N-2}.
\end{proof} 
\begin{lemma}\label{asymptotic-zp-h-s-ln-al}
For all $|z| \leq  K_1$, then there exists $C(K_1)$ such that  $\forall s \geq 1$ we have
\begin{equation}\label{decom-func-h-w-p-lnal}
\left| h(s) |z|^{p-1}z \frac{\ln^{\alpha}(\psi_1^2 z^2 +2)}{\ln^{\alpha}(\psi_1^2 +2)} - \frac{|z|^{p-1} z}{p-1} \right| \leq \frac{ C(K_1)}{s},
\end{equation} 
where $h(s)$ satisfies the asymptotic  \eqref{asym-pto-h-s}.
\end{lemma}
\begin{proof}
We consider $f(z) = \ln^{\alpha}(\psi_1^2z^2 + 2) \forall z \in \mathbb{R}$, then we write
$$ \ln^{\alpha}(\psi_1^2 z^2 +2) = \ln^{\alpha}(\psi_1^2+ 2) + \int_{1}^{|z|}f'(v)dv. $$
Recall from \eqref{asym-pto-h-s} that $h(s) = \frac 1{p-1} + O(\frac 1s)$, we have
\begin{equation}\label{equali-multi-w-p-1-w}
\left| h(s)|z|^{p-1}z\frac{\ln^{\alpha}(\psi_1^2 z^2  + 2)}{\ln^{\alpha}(\psi_1^2 +2)}  - \frac{|z|^{p-1}z}{p-1} \right|  \leq  \frac{C|z|^{p}}{\ln^{\alpha}(\psi_1^2  +2)}\int_{1}^{|z|}|f'(v)|dv   + \frac{ C |z|^p}{s},
\end{equation}
From $i)$ of Lemma \ref{function-h(s)} we have $\frac{1}{\ln(\psi_1^2 + 2)} \leq \frac{C}{s}$, it is  sufficient to show that   
$$A(z):= \frac{|z|^{p}}{\ln^{\alpha -1}(\psi_1^2 +2)} \int_{1}^{|z|} |f'(v)|  dv \leq C(K_1), \quad  \forall |z| \leq K_1,$$
where 
$$f'(v) = \alpha \ln^{\alpha -1 }(\psi_1^2 v^2 +2) \frac{  2 v \psi_1^2}{\psi_1^2 v^2 + 2}.$$
For  $ 1 \leq |z| \leq K_1$, it is trivial to see that $|A(z)| \leq C(K_1).$
For   $|z| <  1$, we consider two cases:

- Case 1: $\alpha - 1 \geq 0$, then
$$A(z) \leq  2|\alpha| |z|^p \int_{|z|}^1 \frac{1}{v} dv \leq C(K_1).$$

- Case 2: $\alpha  -1 < 0$, then
	$$ A(z) \leq  2 |\alpha| |z|^p \frac{\ln^{\alpha -1} (\psi_1^2 z^2  +2)}{\ln^{\alpha -1} (\psi_1  +2 )} \int_{|z|}^1 \frac{1}{v} dv.$$
	+ if  $\psi_1 z^2 \geq 1$ then
	$$A(z) \leq 2 |\alpha|  \frac{ \ln^{1  -\alpha}(\psi_1^2  + 2)}{ \ln^{1 - \alpha} (\psi_1  + 2)} |z|^p \int_{|z|}^1 \frac{1}{v} dv  \leq C(K_1).$$

+ if $\psi_1 z^2 \leq 1$ then $|z| \leq v  \leq \psi_1^{-\frac{1}{2}}$ we deduce that
$$ |A(z)| \leq 2 \vert \alpha \vert \psi_1^{\frac{1-p}{2}} \frac{\ln^{1 - \alpha}(\psi_1^2 +2)}{ \ln^{1  -\alpha}(2)}  |z| \int_{|z|}^1 \leq C(K_1).$$
This concludes  the proof of  Lemma \ref{asymptotic-zp-h-s-ln-al}.
\end{proof}
\begin{lemma}[Control of the nonlinear term D in $S_A(s)$]\label{inside-D}
For all $A \geq 1$, there exists $\sigma_3(A) \geq 1$  such that for all $ s \geq \sigma_{3}(A), q(s) \in S_A(s)$ implies
\begin{equation}\label{D1-intside} 
\forall |y| \leq 2 K \sqrt s,  \quad \left| D(q,s) \right| \leq C(K) \frac{ \ln s (1  +|y|)^4}{s^3}, 
 \end{equation} 
and
 \begin{equation}\label{bound-D}
\|D(q,s)\|_{L^{\infty}(\mathbb{R}^n)}  \leq \frac{C}{s}.
\end{equation}
\end{lemma}
\begin{proof} From the definition \eqref{def:D}  of $D$, let us decompose
$$D(q, s)  = D_1 (q ,s)  +D_2(q,s),$$
where
$$
D_1(q,s) = \left( h(s)  - \frac{1}{p-1} \right) \left( |q + \varphi|^{p-1}(q  +\varphi)  -(q + \varphi) \right), 
$$
$$
D_2(q ,s) = h(s)|q + \varphi|^{p-1}(q + \varphi) L(q + \varphi,s),$$
and   $h(s)$  admits the asymptotic behavior  \eqref{hs},    $L$ is defined  in \eqref{def:L-first}. The proof of \eqref{D1-intside} will follow once the following is proved: for all $\vert y \vert \leq 2K \sqrt s $
\begin{equation}\label{expan-D-1-h-2}
\left|  D_1 - \left( \frac{\alpha (\vert y\vert^2 - 2n)}{4ps^2}  - \frac{\alpha }{s}q  \right)\right| \leq C \frac{(1 + |y|^4)\ln s}{s^3}, 
\end{equation}
and
\begin{equation}\label{expan-D-2-h-2}
\left| D_2  + \left( \frac{\alpha (\vert y\vert^2 - 2n)}{4ps^2}  - \frac{\alpha }{s} q \right) \right| \leq C \frac{(1  +|y|^4)\ln s}{s^3}.
\end{equation}
Let us give a proof of \eqref{expan-D-1-h-2}. From the definition of $S_A(s)$, we note that if $q(s) \in S_A(s)$, then
\begin{align}
\forall y \in \mathbb{R}^n, |q(y,s)| &\leq \frac{ C A^2\ln^2 s(1 + |y|^3)}{s^2}, \label{ineq-q_poli-S-A}\\
\|q(s)\|_{L^{\infty}(\mathbb{R}^n)} &\leq  \frac{C A^2}{\sqrt s}. \label{ineq-q-all-space}
\end{align}
From the definition \eqref{def-varphi} of $\varphi$ and \eqref{ineq-q-all-space}, we see that  for all $|y| \leq 2K\sqrt s$, there exists a positive constant C(K) such that 
\begin{equation}\label{insi-q+varphi-positive}
 0 < \frac{1}{C(K)} \leq  (q + \varphi )(y,s) \leq  C(K).
\end{equation}
Using Taylor expansion and the asymptotic \eqref{hs}, we write 
\begin{equation}\label{expans-D-1-in_var-q}
D_1 (q,s) = \left( -\frac{\alpha}{(p-1)s}  + O\left( \frac{\ln s}{s^2} \right)  \right) \left( \varphi^p- \varphi  + \left( p \varphi^{p-1} -1 \right)q \right) + O\left( q^2\right).
\end{equation}
Using again the definition of $\varphi$ and a Taylor  expansion, we derive 
\begin{align*}
\varphi^p &= 1- \frac{(\vert y\vert^2 - 2n)}{4s}  + O\left( \frac{1 + |y|^4}{s^2}\right), \\
\varphi &=   1 - \frac{(\vert y\vert^2 - 2n)}{4ps}  + O\left( \frac{1 + |y|^4}{s^2}\right),\\
p \varphi^{p-1}  -1 &= p-1- \frac{(p-1)(\vert y\vert^2 - 2n)}{4ps}  + O\left( \frac{1 + |y|^4}{s^2}\right),
\end{align*}
as $s \to +\infty$. Inserting \eqref{ineq-q_poli-S-A} and these estimates into \eqref{expans-D-1-in_var-q} yields \eqref{expan-D-1-h-2}.

We now turn to the proof of \eqref{expan-D-2-h-2}. Recall from \eqref{def:L-first} the definition of $L$,
$$L(q + \varphi, s)  = \frac{2 \alpha  \psi_1^2}{\ln(\psi_1^2  + 2)(\psi_1^2 +2)}( q + \varphi -1)  + \frac{1}{\ln^{\alpha}(\psi_1^2 +2)}\int_{1}^{q + \varphi}  f''(v) (q + \varphi -v)dt,
$$
where $f(v) = \ln^{\alpha}(\psi_1^2 v^2  +2), v \in \mathbb{R}$. From  \eqref{insi-q+varphi-positive} and a direct computation, we estimate
$$
\left|  \frac{1}{\ln^{\alpha}(\psi_1^2  + 2)} \int_{1}^{q + \varphi} f''(v) (q  + \varphi  - v) dv\right| \leq C(K)  \frac{|q + \varphi -1|^2}{s},
$$
which yields
\begin{equation}\label{eurer-L-q-varphi}
\left|L(q + \varphi, s) - \frac{2 \alpha \psi_1^2 (q + \varphi  - 1)}{\ln(\psi_1^2 +2)(\psi_1^2  +2)} \right|  \leq C(K) \frac{|q + \varphi -1|^2}{s}.
\end{equation}
From \eqref{lnpsi(s)} and \eqref{eurer-L-q-varphi}, we then have
$$
\left|L(q + \varphi, s) - \frac{\alpha (p-1)  (q + \varphi  - 1)}{s} \right| \leq C(K) \left( \frac{|q + \varphi -1|^2}{s}  + \frac{\ln s \vert q + \varphi -1 \vert }{s^2}\right),
$$
and beside that we have
$$  \vert q + \varphi - 1 \vert \leq \frac{C ( 1 + \vert y \vert^2 }{s},$$
imply that
\begin{equation}\label{eurer-L-q-varphi-1}
\left|L(q + \varphi, s) - \frac{\alpha (p-1)  (q + \varphi  - 1)}{s} \right| \leq C(K) \frac{\ln s(1 + \vert y\vert^4)}{s^3},
\end{equation}
Moreover, from definition of $D_2 $ and \eqref{eurer-L-q-varphi-1} we deduce that
$$
\left| D_2(q,s) - \frac{\alpha}{s} \left(  \varphi^{p+1} - \varphi^p  + ((p+1)\varphi^p  - p\varphi^{p-1} )q  \right) \right| \leq C\frac{(1  +|y|^4) \ln s}{s^3}, 
$$
and
\begin{align*}
\varphi^{p+1} - \varphi^p  &= - \frac{(\vert y\vert^2 - 2)}{4ps} + O \left( \frac{1 + |y|^4}{s^2}\right), \text{ as },\\
(p+1) \varphi^{p} -   p\varphi^{p-1} &= 1 - \frac{(\vert y\vert^2 - 2)}{2s}  + O\left( \frac{1 + |y|^4}{s^2}\right), \text{ as },
\end{align*}
as $s \to +\infty$ which yield \eqref{expan-D-2-h-2}.

We now prove for \eqref{bound-D}. 
From \eqref{hs} and the boundedness of $q$ and $\varphi$, we have 
$$|D_1(q,s)| \leq \frac Cs.$$
It is sufficient to prove that for all $y \in \mathbb{R}^n$,
$$|D_2(q,s)| \leq \frac{C(K) }{s},$$
Indeed, from definition \eqref{def:L-first} of $L$ we deduce that
$$D_2(q,s) = h(s) |q + \varphi|^{p-1}(q + \varphi) \frac{\ln^{\alpha}(\psi_1^2 z^2 +2)}{\ln^{\alpha}(\psi^2 +2)} - h(s)|q + \varphi|^{p-1}(q + \varphi).$$
Using Lemma \ref{asymptotic-zp-h-s-ln-al} we deduce
$$|D_2(q,s)| \leq \frac{C(K)}{s}.$$
This completes the proof of Lemma \ref{inside-D}.
\end{proof}

\begin{lemma} \label{lemma-of-rest-termes} When $s$ large enough, then we  have for all $y \in \mathbb{R}^n$:
\begin{itemize}
\item[$i)$] (Estimates on $V$): 
$$\left| V (y,s)\right| \leq \frac{C (1 + |y|^2)}{s}, \forall y \in \mathbb{R}^n, $$
and
$$ V = -  \frac{(|y|^2  - 2n)}{4s} + \tilde V \quad \text{with} \quad \tilde V = O \left( \frac{1 + |y|^4}{s^2} \right), \forall |y| \leq K \sqrt s.$$
\item[$ii)$] (Estimates on $R$ ) 
$$| R(y,s) | \leq \frac{C}{s}, \forall y \in \mathbb R^n,$$
and 
$$ R(y,s) = \frac{c_p}{s^2}  +  \tilde R(y,s) \quad \text{with} \quad \tilde R  = O \left(  \frac{1 + |y|^4}{s^3}\right), \forall |y| \leq K \sqrt s.$$
\end{itemize}
\end{lemma}
 \begin{proof} The proof simply follows from Taylor expansion. We refer to Lemmas B.1 and B.5 in \cite{ZAAihn98} for a similar proof. 
 \end{proof}
 \begin{lemma}[Estimates on $B(q)$]\label{estimate-B-q}
 For all $A > 0$  there exists $\sigma_5(A) > 0$ such that for all $ s \geq \sigma_5 (A), q(s) \in S_A(s)$ implies 
 \begin{equation}\label{bound-ourside-B-q}
 | B (q (y,s))|  \leq C |q|^2,
 \end{equation}
 and
 \begin{equation}\label{bound-allspace-B-q}
 |B(q)|  \leq C |q|^{\bar p},
 \end{equation}
 with $\bar p  = \min (p,2)$.
 \end{lemma}
 \begin{proof} See Lemma 3.6 in \cite{MZdm97} for the proof of this lemma.
 \end{proof}



\def\cprime{$'$}

\end{document}